\documentclass{article}
\usepackage{amsmath}		
\usepackage{amscd}
\usepackage{amsthm}		
\usepackage{amsfonts}
\usepackage{amssymb}		
\usepackage{mathtools}		

\usepackage[english]{babel}		
\usepackage[utf8]{inputenc}	
\usepackage{calc}
\usepackage{thumbpdf}			
\usepackage{latexsym}
\usepackage{xcolor}
\usepackage{enumerate}		
\usepackage{url}				
\usepackage{cite}	
\usepackage{graphics} 
\usepackage{epsfig} 
\usepackage{color}

\usepackage{ifthen}

\usepackage[bookmarks,bookmarksnumbered,colorlinks]{hyperref}
\hypersetup{linkcolor = black,anchorcolor = black,citecolor =
	black,filecolor = black,urlcolor = black}

\definecolor{darkgreen}{rgb}{0,0.6,0}
\definecolor{lessdarkgreen}{rgb}{0.4,0.85,0.1}

\newtheorem{defi}{Definition}[section]
\newtheorem{lem}[defi]{Lemma}
\newtheorem{thm}[defi]{Theorem}
\newtheorem{coro}[defi]{Corollary}

\newtheorem{ass}[defi]{Assumption}
\newtheorem{rem}[defi]{Remark}

\usepackage{algorithm}
\usepackage[noend]{algpseudocode}

\DeclareMathOperator*{\argmin}{arg\,min}
\DeclareMathOperator*{\minimize}{minimize}
\def\eps{\varepsilon}

\newcommand{\R}{\mathbb{R}}
\newcommand{\N}{\mathbb{N}}

\newcommand{\K}{\mathcal{K}}
\newcommand{\U}{\mathcal{U}}
\newcommand{\X}{\mathcal{X}}
\newcommand{\Y}{\mathcal{Y}}

\newcommand{\UU}{\mathbb{U}}

\newcommand{\W}{\mathcal{W}}
\newcommand{\B}[1]{\mathcal{B}(#1)}
\newcommand{\RR}[1]{\mathcal{R}(#1)}

\newcommand{\F}{\mathcal{F}}

\newcommand{\Prob}{\mathbb{P}}
\newcommand{\E}{\mathbb{E}}

\newcommand{\Exp}[1]{\mathbb{E}\left[#1\right]}

\newcommand{\tol}{k}
\newcommand{\tcl}{j}
\newcommand{\Bcl}{\bar{B}}

\usepackage{setspace}
\usepackage[margin=2.5cm]{geometry}

\providecommand{\keywords}[1]
{
  \small	
  \textbf{\textit{Keywords---}} #1
}
\providecommand{\MSCcodes}[1]
{
  \small	
  \textbf{\textit{MSC 2020 codes---}} #1
}

\title{Stability and performance of stochastic economic MPC \\
-- Stochastic characterization of the closed-loop asymptotics
        \thanks{This work was funded by the Deutsche Forschungsgemeinschaft (DFG, German Research Foundation) -- project number 499435839.}
        }
\author{Jonas Schießl \thanks{Mathematical Institute, University of Bayreuth, Bayreuth, Germany (e-mail: jonas.schiessl@uni-bayreuth.de).}
    \and Hannah Selder \thanks{Center for Scalable Data Analytics and Artificial Intelligence (ScaDS.AI) Dresden/Leipzig, Leipzig University, Leipzig, Germany (e-mail: hannah.selder@uni-leipzig.de).}
    \and Ruchuan Ou \thanks{Institute of Control Systems, Hamburg University of Technology, Hamburg, Germany (e-mail: ruchuan.ou@tuhh.de).}
    \and Michael H. Baumann \thanks{Mathematical Institute, University of Bayreuth, Bayreuth, Germany (e-mail: michael.baumann@uni-bayreuth.de).}
    \and Timm Faulwasser \thanks{Institute of Control Systems, Hamburg University of Technology, Hamburg, Germany (e-mail: timm.faulwasser@ieee.org).}
    \and Lars Grüne \thanks{Mathematical Institute, University of Bayreuth, Bayreuth, Germany (e-mail: lars.gruene@uni-bayreuth.de).}
    }

\begin{document}

\maketitle

\doublespacing

\begin{abstract}
    Model Predictive Control (MPC) is well understood in the deterministic setting, yet rigorous stability and performance guarantees for stochastic MPC remain limited to the consideration of terminal constraints and penalties. In contrast, this work analyzes stochastic economic MPC with an expected cost criterion and establishes closed-loop guarantees without terminal conditions. Relying on stochastic dissipativity and turnpike properties, we construct closed-loop Lyapunov functions that ensure $P$-practical asymptotic stability of a particular optimal stationary process under different notions of stochastic convergence, such as in distribution or in the $p$-th mean. In addition, we derive tight near-optimal bounds for both averaged and non-averaged performance, thereby extending classical deterministic results to the stochastic domain. Finally, we show that the abstract stochastic MPC scheme requiring distributional knowledge shares the same closed-loop properties as a practically implementable algorithm based only on sampled state information, ensuring applicability of our findings. Our findings are illustrated by a numerical example.
\end{abstract}


\MSCcodes{93B45, 93D15, 93E15, 93E20}

\section{Introduction} \label{sec:introduction}

Model Predictive Control (MPC) is a well-established methodology for the control of dynamical systems, with solid theoretical foundations in the deterministic setting regarding stability, feasibility, and performance. Motivated by the inherent uncertainties in real-world systems, recent research has aimed to extend these concepts to stochastic MPC \cite{Mesbah2016}.

However, while the stability and performance of deterministic MPC for general (often also called ``economic'') cost criteria are by now well understood \cite{DiAR10, Angeli2012, Gruene2013, Faulwasser2018, Gruene2019, CoGW20}, rigorous guarantees in the stochastic domain are still rare. 

One line of existing works tackles stochastic MPC resorting to the analysis tools originally developed for robust approaches, see, e.g., \cite{Lorenzen16a}. Put differently, the randomness of state evolution and cost accumulation makes stability and performance analysis substantially more challenging and most existing results are limited.
For instance, most of the existing stability results require terminal ingredients, e.g., in \cite{Lucia20a}, where input-to-state practical stability of a multi-stage MPC Scheme with terminal conditions is shown, in \cite{McAR23}, where a robust asymptotic stability property in expectation is shown for a general class of problems, in \cite{Kouvaritakis16a}, which provides different stability results for linear systems with terminal constraints, and in \cite{Chatterjee15a}, which shows stability properties based on drift conditions. In a recent line of research, probabilistic reachable sets are used for constraint tightening, see, e.g.,~\cite{Hewing18a,Kohler25}, which allows to handle chance constraints.

However, one of the most striking limitations is that---to the best of our knowledge---existing stability results only provide asymptotic bounds for certain averaged quantities (e.g., in the form of input-to-state stability estimates), but they do not give a precise description of the limit behavior of the MPC closed loop in terms of the corresponding stochastic process or its distribution.
In terms of performance analysis, to the best of the authors' knowledge the only results available use non-tight bounds on the expected average performance as, e.g., in \cite{Chatterjee15a} and a dynamic regret analysis for linear-quadratic problems with additive and multiplicative uncertainties under a finite support assumption \cite{shin2022}.

In this paper, we contribute to this line of research by analyzing the closed-loop behavior of discrete-time stochastic MPC with a general expected cost criterion, defining a setting that is often called ``economic MPC''.
Our theoretical findings are based on dissipativity and turnpike concepts for stochastic optimal control problems developed in \cite{Schiessl2023a,Schiessl2024a,Schiessl2025}.
Building on these structural insights, we construct Lyapunov functions that allow us to establish practical asymptotic stability of the MPC closed-loop system without requiring terminal ingredients. In contrast to existing results, we provide an exact description of the asymptotic behavior of the MPC closed-loop solution in terms of a particular stationary stochastic process. This enables us to provide stability statements on different levels, such as mean square, pathwise in probability, in distribution, or in moments, which are in line with common stochastic stability and convergence concepts, such as convergence in the $p$-th mean or convergence in distribution.
Moreover, we demonstrate that the closed-loop trajectory enjoys averaged and non-averaged near-optimal performance bounds, thereby extending classical results from deterministic economic MPC to the stochastic domain.
Preliminary results for the performance analysis under more restrictive assumptions are contained in the authors’ recent conference paper \cite{Schiessl2024b}.

Due to the conceptual approach we take in this paper, in a first step the results apply to an abstract stochastic MPC algorithm that requires the knowledge of the distribution of the closed-loop solutions, which is usually not available in practice. Yet, a result that we consider interesting on its own shows that this abstract algorithm has the same statistical closed-loop properties as an MPC algorithm that only uses state information sampled along a single closed-loop path and is thus practically implementable. 

The remainder of the paper is organized as follows.
Section~\ref{sec:setting} introduces the problem set-up and defines the abstract stochastic MPC algorithm that forms the basis of our theoretical considerations.
In Section~\ref{sec:stochDPP}, we present stochastic dynamic programming principles and show that the abstract stochastic MPC algorithm analyzed here possesses the same closed-loop properties as a practically implementable scheme based solely on sampled state information, thereby ensuring the relevance of our results for applications.
Section~\ref{sec:DissiTurnpike} introduces the notions of stochastic dissipativity and turnpike properties used for our theoretical developments and establishes the connection between them.
In Section~\ref{sec:Stability} and Section~\ref{sec:Performance}, we present the main results of the paper. In particular in Section~\ref{sec:Stability}, we demonstrate that stochastic dissipativity enables the construction of Lyapunov functions associated with different notions of stochastic convergence, thereby guaranteeing the stochastic stability of the closed-loop trajectories. Furthermore, based on these stability results in Section~\ref{sec:Stability} we derive tight near-optimality bounds for the MPC closed loop with respect to both averaged and non-averaged performance.
Our results are illustrated by a numerical example in Section~\ref{sec:numerics} and Section~\ref{sec:conclusion} concludes the paper.

\section{Setting} \label{sec:setting}

\subsection{Problem formulation}

    For separable Banach spaces $\X$, $\U$, and $\W$ and a continuous function
    \begin{equation*}
        f : \X\times \U \times \W \rightarrow \X, \quad (x,u,w) \mapsto f(x,u,w)
    \end{equation*}
    we consider the discrete-time stochastic system
    \begin{equation} 
    \label{eq:stochSys}
        X(\tol+1) = f(X(\tol),U(\tol),W(\tol)), \quad X(0) = X_0.
    \end{equation}
    Here, the initial condition $X_0 \in \RR{\Omega,\X}$, the states $X(\tol) \in \RR{\Omega,\X}$, the controls $U(\tol) \in \RR{\Omega,\U}$ and the noise $W(\tol) \in \RR{\Omega,\W}$ are from the spaces
    \[\RR{\Omega,\Y} := \{X: (\Omega,\mathcal{F},\Prob) \rightarrow (\Y,\mathcal{B}(\Y)) \mbox{ measurable} \}, \qquad \Y=\X,\,\U,\,\mbox{or }\W\]
    of random variables on the probability space $(\Omega, \mathcal{F}, \Prob)$ for all $\tol \in \N_0$,
    and $\B{\Y}$ is the Borel $\sigma$-algebra on $\Y$.
    We assume that the sequence $\{W(\tol)\}_{\tol \in \N_0}$ is \emph{i.i.d.}, $W(\tol)$ is independent of $X(\tol)$ and $U(\tol)$ for all $\tol \in \N_0$, and that $W(\tol) \sim P_W$, i.e., $W(\tol)$ has the distribution $P_W$.
    
    Furthermore, we consider control process $\mathbf{U} := (U(0),U(1),\ldots)$ that are measurable with respect to the natural filtration $(\F_\tol)_{\tol \in \N_0}$, i.e.\
    \begin{equation} \label{eq:Filtration}
        \sigma(U(\tol)) \subseteq \F_\tol := \sigma((X(0),\ldots,X(\tol)), \quad \tol \in \N_0,
    \end{equation}
    for all $\tol \in \N_0$. 
    This condition can be interpreted as a causality requirement, formalizing that when selecting $U(k)$, only past and present information is used, without access to future events. 
    Moreover, this choice of the filtration restricts the control design to information that is observable through the states. 
    Consequently, the same control input must be applied whenever two different noise realizations lead to the same state.
    Notice that when a filtration is generated directly by the noise process, the inputs applied at the resulting does not necessarily need to coincide.
    For further details on stochastic filtrations, we refer to \cite{Fristedt1997, Protter2005}.
    In addition to the filtration condition we impose pathwise constraints on the control process, i.e., for a set $\UU \subseteq \U$ the condition  $U(\omega, \tol) \in \UU$ should be satisfied for all $\omega \in \Omega$ and $\tol \in \N_0$.

    To extend our dynamics \eqref{eq:stochSys} to an optimal control problem, we consider a lower semicontinuous function $g: \X \times \U \rightarrow \R$ bounded from below and define the stage cost $\ell(X,U) := \Exp{g(X,U)}$, where $\Exp{g(X,U)}$ denotes the expectation of the function $g$ with respect to the joint distribution of the state-control-pair $(X,U)$.
    
    Then the stochastic optimal control problem on time horizon $N \in \N \cup \{\infty\}$ under consideration reads
    \begin{equation} 
    \label{eq:stochOCP}
        \begin{split}
            \minimize_{\mathbf{U}} &~J_N(X_0,\mathbf{U}) := \sum_{\tol=0}^{N-1} \ell(X(\tol),U(\tol)) \\
            s.t. ~ X(\tol+1) &= f(X(\tol),U(\tol),W(\tol)), ~ X(0) = X_0, \\
            \sigma(U(\tol)) &\subseteq \F_\tol, \quad U(\omega,\tol) \in \UU,~\text{for all}~\omega \in \Omega.
        \end{split}
    \end{equation}

    Our goal is now to approximate the solution of the stochastic optimal control problem \eqref{eq:stochOCP} on infinite horizon $N=\infty$ while satisfying near-optimality and stability properties. 

\subsection{The abstract stochastic MPC algorithm}

    To calculate an approximation of the solution of \eqref{eq:stochOCP} on horizon $N = \infty$ we will use a stochastic MPC scheme.
    The idea of MPC is to approximate the solution to problem~\eqref{eq:stochOCP} on infinite horizon by iteratively solving problems on finite horizons $N \in \N$ at each time $\tcl \in \N_0$.
    While in time-invariant deterministic settings it is sufficient to change the initial condition of the optimal control problem at each time $\tcl \in \N_0$, this is not sufficient for our setting since we have an implicit time dependence of the problem caused by the noise and since the filtration condition \eqref{eq:Filtration} depends on the whole history of states up to time $\tcl$.
    Yet, it is a well known fact in stochastic optimal control, cf.\ \cite{Altman2021,Bertsekas1996b}, that the optimal solution $\mathbf{U}^*_N$ of problem~\eqref{eq:stochOCP} satisfies
    $$\sigma(U^*_N(\tol)) \subseteq \sigma(X^*(\tol)) \subseteq \F_\tol,$$
    and thus, we could replace the condition $\sigma(U(\tol)) \subseteq \F_\tol$ in \eqref{eq:stochOCP} by  $\sigma(U(\tol)) \subseteq \sigma(X(\tol))$ and still obtain the same optimal solution.
    However, while this eliminates the dependence of the history of states of the problem, it is still time-varying due to the influence of the sequence of disturbances $\mathbf{W}$, at least if we are working on the space of random variables.
    Thus, we interpret the system~\eqref{eq:stochSys} as a time-varying system on the infinite dimensional space of random variables $\RR{\Omega,\X}$. For a given initial state $X_\tcl \in \RR{\Omega,\X}$ at time $\tcl \in \N_0$ and control sequence $\mathbf{U}$ we denote by $X_U(\cdot;\tcl,X_\tcl)$ the solution of the system
    \begin{equation} 
    \label{eq:stochSysTimeVarying}
        X_U(\tol+1;\tcl,X_\tcl) = f(X_U(\tol;\tcl,X_\tcl),U(\tol),W(\tcl+\tol)), \quad X(0;\tcl,X_\tcl) = X_\tcl,
    \end{equation}
    using the abbreviations $X(\cdot;\tcl,X_\tcl)$ if the control is unambiguous and $X(\cdot)$ if the initial time and state is also unambiguous.
    The open-loop problem solved at each time instant $j$ in our stochastic MPC algorithm, associated to the stochastic optimal control problem~\eqref{eq:stochOCP} is then given by
    \begin{equation} 
    \label{eq:stochOCPopenloop}
        \begin{split}
            \minimize_{\mathbf{U}} &~J_N(X_\tcl,\mathbf{U}) := \sum_{\tol=0}^{N-1} \ell(X(\tol;\tcl,X_\tcl),U(\tol)) \\
            s.t. ~ X(\tol+1;\tcl,X_\tcl) &= f(X(\tol;\tcl,X_\tcl),U(\tol),W(\tcl+\tol)), ~ X(0;\tcl,X_\tcl) = X_\tcl \\
            \sigma(U(\tol)) &\subseteq \sigma(X(\tol;\tcl,X_\tcl)), \quad U(\omega,\tol) \in \UU,~\text{for all}~\omega \in \Omega.
        \end{split}
    \end{equation}
    
    If a minimizer of this problem exists, we will denote it by $\mathbf{U}^*_N$ or $\mathbf{U}^*_{N,\tcl,X_\tcl}$ if we want to emphasize the dependence on the initial condition. 
    Furthermore, we define the set of admissible control sequences to the initial time-state pair $(\tcl,X_\tcl)$ on horizon $N \in \N \cup \{\infty\}$ as 
    $$\UU_{ad}^N(\tcl,X_\tcl) := \{ \mathbf{U} \in \RR{\Omega,\U} \mid U(\tol) \in \UU_{ad}(X_{\mathbf{U}}(\tol;\tcl,X_\tcl)) \mbox{ for all } 0 \leq \tol \leq N-1 \}$$
    with $\UU_{ad}(X) := \{U \in \RR{\Omega,\X} \mid \sigma(U) \subseteq \sigma(X) ,~ U(\omega) \in \UU ,~ \omega \in \Omega \}$ and by 
    \[ V_N(\tcl,X_\tcl) := \inf_{\mathbf{U} \in \UU_{ad}^{N}(\tcl,X_\tcl)} J_N(X,\mathbf{U}) \] 
    we denote the optimal value function corresponding to problem~\eqref{eq:stochOCPopenloop}.
    
    The abstract version of the stochastic MPC algorithm on the space of random variables is summarized in Algorithm~\ref{alg:abstractStochMPC}.
    Furthermore, we recall that every optimal control sequence obtained in Step~2 of the algorithm is generated by a deterministic sequence of state-dependent feedback policies.
    
    \begin{algorithm}
        \caption{Abstract stochastic MPC algorithm}
        \label{alg:abstractStochMPC}
        \begin{algorithmic}
            \For{$\tcl=0,1,\ldots$}
                \State 1.) Set $X_\tcl = X(\tcl)$.
                \State 2.) Solve the stochastic optimal control problem \eqref{eq:stochOCPopenloop} and obtain the optimal control sequence
                $$\mathbf{U}^*_{N,\tcl,X_\tcl} = (U^*_{N,\tcl,X_\tcl}(0), \ldots, U^*_{N,\tcl,X_\tcl}(N-1)) = (\pi^*_0(X(0;\tcl,X_\tcl)), \ldots, \pi^*_{N-1}(X(N-1;\tcl,X_\tcl))).$$
                \State 3.) Apply the MPC feedback $\mu_N(X(\tcl)) := U^*_{N,\tcl,X_\tcl}(0)$ to system \eqref{eq:stochSys} 
                \State ~~~ and get the next state $X(\tcl+1)$.
            \EndFor
        \end{algorithmic}
    \end{algorithm}

    \begin{rem}[Link to MDPs]
        Note that due to the law-invariant structure of problem~\ref{eq:stochOCPopenloop}, it is also possible to reformulate the optimal control problem~\eqref{eq:stochOCPopenloop} as a Markov decision problem (MDP), defined on the space of probability measures and feedback laws.
        In this reformulation, the problem becomes time-invariant, since the random variables $W(0), W(1), \ldots$, which model the noise, are \emph{i.i.d.}
        While this implies that the feedback law $\mu_N$ is time-invariant, such a reformulation allows us to capture only the distributional behavior of the solutions, since information about the realization paths is lost.
        Hence, we will not employ this reformulation explicitly, as our analysis is carried out in the space of random variables to make statements about both the pathwise behavior and the behavior of the underlying probability measures.
    \end{rem}
    
    While the abstract MPC Algorithm~\ref{alg:abstractStochMPC} on the space of random variables is useful for our theoretical analysis, its practical relevance is limited. This is because in a real plant one can usually only measure the value of the current realization of a random variable and not its distribution. 
    However, in the following section we will show that under the assumed cost structure there exists a implementable MPC algorithm, which only needs information on single realization paths, yet the resulting closed-loop solutions coincide almost surely with those generated by Algorithm~\ref{alg:abstractStochMPC}.

\section{Stochastic dynamic programming and an implementable version of the abstract MPC algorithm}

An important tool to analyze optimal control problems is the dynamic programming principle (DPP). It states that minimizing the sum of the cost on a shorter horizon $M < N$ plus the optimal value on the remaining horizon yields the same optimal value as directly minimizing the cost on the whole horizon. 
    The following two theorems formalize this principle on finite and infinite horizon.

    \begin{thm}[{Finite horizon DPP, cf.\ \cite[Theorem~3.15]{Gruene2017b}}] \label{thm:FiniteDPP}
       Consider the optimal control problem \eqref{eq:stochOCPopenloop} with $\tcl,N \in \N_0$ and $X_\tcl \in \RR{\Omega,\X}$ such that $\vert V_N(\tcl,X_\tcl) \vert < \infty$.
       Let $\mathbf{U}^*_N \in \UU_{ad}^N(\tcl,X_\tcl)$ be an optimal control sequence on horizon $N$ and define $V_0 \equiv 0$. 
       Then for all $M=1,\ldots,N$ it holds that 
       \begin{equation*}
           V_N(\tcl,X_\tcl) = \sum_{\tol=0}^{M-1} \ell(X_{\mathbf{U}^*_N}(\tol;\tcl,X_\tcl), U^*_N(\tol)) + V_{N-M}(\tcl+M,X_{\mathbf{U}^*_N}(M;\tcl,X_\tcl)).
       \end{equation*}
    \end{thm}

    \begin{thm}[{Infinite horizon DPP, cf.\cite[Theorem~4.4]{Gruene2017b}}] \label{thm:InfiniteDPP}
       Consider the optimal control problem \eqref{eq:stochOCP} with $\tcl \in \N_0$, $N=\infty$ and $X_\tcl \in \RR{\Omega,\X}$ such that $\vert V_{\infty}(\tcl,X_\tcl) \vert < \infty$. 
       Let $\mathbf{U}^*_{\infty} \in \UU_{ad}^{\infty}(\tcl,X_\tcl)$ be an optimal control sequence on infinite horizon. 
       Then for all $M \in \N$ it holds that
       \begin{equation*}
           V_{\infty}(\tcl,X_\tcl) = \sum_{\tol=0}^{M-1} \ell(X_{\mathbf{U}^*_{\infty}}(\tol;\tcl,X_\tcl), U^*_{\infty}(\tol)) + V_{\infty}(\tcl+M,X_{\mathbf{U}^*_{\infty}}(M;\tcl,X_\tcl)).
       \end{equation*}
    \end{thm}

    While these theorems are stated in terms of random variables---i.e., the optimal solutions are calculated for all possible realizations simultaneously---one can show that the DPP holds pathwise for each realization.
    We recall that a function is called universally measurable if it is measurable with respect to the universal $\sigma$-algebra $\cap_{\mathbb{Q} \in \mathcal{M}(\X)} \mathcal{B}_{\mathbb{Q}}(\X)$ with $\mathcal{B}_{\mathbb{Q}}(\X) := \{ E \subseteq \X \mid \mathbb{Q}(E) + \mathbb{Q}^*(E^c) = 1 \}$, where $\mathcal{M}(\X)$ is the space of probability measures on $\X$ and $\mathbb{Q}^*$ denotes the $\mathbb{Q}$-outer measure of $\mathbb{Q}$, cf.\ \cite[Equation~(87)]{Bertsekas1996b}.

    \begin{thm}[{cf.\ \cite[Proposition~8.5]{Bertsekas1996b}}]
    \label{thm:PointwiseDPP}
        Consider $\tcl \in \N_0$ and assume that for $\tol=0,\ldots,N-1$ the infimum in 
        \begin{equation} \label{eq:DPPpath}
            \inf_{u_\tol \in \U} \{ g(x_\tol,u_\tol) + V_{N-\tol-1}(\tcl+\tol,f(x_\tol,u_\tol,W(\tcl+\tol))) \}
        \end{equation}
        is achieved for each $x_\tol \in \X$. Define $\pi_\tol^*(x_\tol):=u_\tol^*$ for each $x_\tol\in\X$, where the $u_\tol^*\in\U$ form a measurable selection of minimizers of \eqref{eq:DPPpath}. 
        Then $\boldsymbol{\pi}^* := (\pi_0^*,\ldots,\pi_{N-1}^*)$ defines a sequence of universally measurable functions and 
        \begin{equation*}
            \pi^*_k \in \argmin_{\pi_\tol \in \Pi} \{ \ell(X,\pi_\tol(X)) + V_{N-\tol-1}(\tcl+\tol,f(X,\pi_\tol(X),W(\tcl+\tol))) \}
        \end{equation*}
        holds
        for all $X \in \RR{\Omega,\X}$ with $\vert V_N(\tcl,X) \vert < \infty$ and $\tol=0,\ldots,N-1$, where $\Pi$ is the set of all universally measurable functions $\pi: \X \rightarrow \U$.
    \end{thm}

    Theorem~\ref{thm:PointwiseDPP} shows that we can perform the calculations from Algorithm~\ref{alg:abstractStochMPC} pathwise for each realization, which leads to Algorithm~\ref{alg:implementableStochMPC}.
    Note that this algorithm can be implemented in a real plant and is therefore of particular practical relevance, as it works without knowledge of the full state distribution, but solely with measurements of the realizations.
    The full feedback law from Algorithm~\ref{alg:abstractStochMPC} would then be regained by performing Algorithm~\ref{alg:implementableStochMPC} for every possible realization.
    Note that, in contrast to Algorithm~\ref{alg:abstractStochMPC}, the MPC feedback obtained in each iteration is not only induced by a deterministic feedback policy, but is in fact a deterministic value, since the initial condition $x_j \in \X$ is deterministic.

    \begin{algorithm}
        \caption{Implementable stochastic MPC algorithm}
        \label{alg:implementableStochMPC}
        \begin{algorithmic}
            \For{$\tcl=0,1,\ldots$}
                \State 1.) Measure the state $x_\tcl := X(\tcl,\omega)$ and set $X_0 \equiv x_\tcl$.
                \State 2.) Solve the stochastic optimal control problem \eqref{eq:stochOCP} and obtain the optimal control sequence
                $$\mathbf{U}^*_{N,\tcl,x_\tcl} = (U^*_{N,\tcl,x_\tcl}(0), \ldots, U^*_{N,\tcl,x_\tcl}(N-1)) = (\pi^*_0(X(0;\tcl,x_\tcl)), \ldots, \pi^*_{N-1}(X(N-1;\tcl,x_\tcl))).$$
                \State 3.) Apply the MPC feedback $\mu_N(x_\tcl) := U^*_{N,x_\tcl}(0)$ to system \eqref{eq:stochSys}.
            \EndFor
        \end{algorithmic}
    \end{algorithm}

    The remaining problem is that Theorem~\ref{thm:PointwiseDPP} only guarantees that the control sequence which is computed in that way is \emph{universally} 
    measurable, cf.\ \cite[Definition~7.18]{Bertsekas1996b}, and thus may not be a feasible solution of problem~\eqref{eq:stochOCPopenloop}.
    Hence, we must ensure that we can also construct a measurable control sequence out of the pathwise solutions from Theorem~\eqref{thm:PointwiseDPP}. This is covered by the following lemma.

    \begin{lem}
    \label{lem:measurable}
        For every universally measurable function $\pi: \X \rightarrow \U$ and every $\mathbb{Q} \in \mathcal{M}(\X)$ there is a measurable function $\pi_{\mathbb{Q}} : \X \rightarrow \U$ such that $\pi = \pi_{\mathbb{Q}}$ holds $\mathbb{Q}$-almost surely.
    \end{lem}
    \begin{proof}
        Since $\pi$ is universally measurable, it is measurable with respect to the completed Borel $\sigma$-algebra 
        $\mathcal{B}_{\mathbb{Q}}(\X)$
        given any $\mathbb{Q} \in \mathcal{M}(\X)$. Thus, the claim directly follows by applying \cite[Lemma~1.2]{Crauel2002}.
    \end{proof}

    Now we can use Lemma~\ref{lem:measurable} to conclude that the closed-loop trajectories and performance of Algorithm~\ref{alg:abstractStochMPC} and Algorithm~\ref{alg:implementableStochMPC} coincide almost-surely.
    Here, we measure the closed-loop performance of a measurable feedback law $\mu: \X \to \U$ over a horizon $K \in \N$ by
    \[ J^{cl}_K(X_0, \mu) := \sum_{\tcl=0}^{K-1} \ell(X_{\mu}(\tcl;0,X_0),\mu(X_{\mu}(\tcl;,0,X_0)))\]
    where $X_{\mu}(\cdot,0,X_0)$ denotes the closed-loop trajectory defined by system \eqref{eq:stochSysTimeVarying} with control $U(\tol) = \mu(X_{\mu}(\tol))$ and initial condition $X_{\mu}(0) = X_0$.

    \begin{coro}
    \label{cor:mpcalg}
        Consider $N$, $\tcl \in \N$, and initial state $X_0 \in \RR{\Omega,\X}$ and let $\mu_N^{a}$ be the feedback law from the abstract MPC Algorithm~\ref{alg:abstractStochMPC}. 
        Then there exists a measurable feedback law $\mu_N^{i}$ generated pathwisely by the implementable Algorithm~\ref{alg:implementableStochMPC} such that 
        \begin{equation} \label{eq:solutionsCoincide}
             \mu_N^{a}(X_{\mu_N^{a}}(\tol)) = \mu_N^{i}(X_{\mu_N^{i}}(\tol)), \quad
             X_{\mu_N^{a}}(\tol) = X_{\mu_N^{i}}(\tol)
        \end{equation}
        holds $\mathbb{P}$-almost surely for all $\tol \in \N_0$ and the closed-loop cost satisfies
        $$J^{cl}_K(X_0, \mu_N^{a}) =  J^{cl}_K(X_0, \mu_N^{i}).$$
    \end{coro}
    \begin{proof}
        Consider that $X_{\mu_N^{a}}(\tol) = X_{\mu_N^{i}}(\tol)$ for some $\tol \in \N_0$. Then using Theorem~\ref{thm:PointwiseDPP} in combination with Lemma~\ref{lem:measurable} we obtain that $\mu_N^{a}(X_{\mu_N^{a}}(\tol)) = \mu_N^{i}(X_{\mu_N^{i}}(\tol))$ holds almost surely, which also yields $X_{\mu_N^{a}}(\tol+1) = X_{\mu_N^{i}}(\tol+1)$. 
        Hence, equation~\eqref{eq:solutionsCoincide} follows by induction since the initial condition is fixed and thus $X_{\mu_N^{a}}(0) = X_{\mu_N^{i}}(0) = X_0$ holds per definition.
        Furthermore, since the stage costs $\ell(X,U) = \Exp{g(X,U)}$ are only depending on the joint distribution of the pair $(X,U)$, the identity for the costs is directly implied by equation~\eqref{eq:solutionsCoincide}.
    \end{proof}

    \begin{rem}
        \begin{enumerate}[(i)]
            \item As our proof techniques only yield properties of Algorithm~\ref{alg:abstractStochMPC} directly,
            Corollary \ref{cor:mpcalg} is pivotal for our analysis, as it guarantees that the results we will derive in the remainder of this paper for the MPC closed-loop trajectory generated by the abstract MPC Algorithm~\ref{alg:abstractStochMPC} will also be valid for the MPC closed-loop trajectories generated by the practically implementable MPC Algorithm~\ref{alg:implementableStochMPC}. 
            \item As we do not assume uniqueness of the optimal controls, Algorithm~\ref{alg:abstractStochMPC} may not yield a unique MPC closed-loop trajectory. However, the estimates we derive below are valid for all possible closed-loop trajectories.
            \item In practice, only finitely many evaluations of $\mu_N=\mu_N^i$ will be performed in Algorithm~\ref{alg:implementableStochMPC}. Since a finite selection of values is always measurable, the requirement of $\mu_N^i$ being measurable does not have any practical implications when running Algorithm~\ref{alg:implementableStochMPC}.
        \end{enumerate}
    \end{rem} \label{sec:stochDPP}

\section{Stochastic dissipativity and turnpikes} \label{sec:DissiTurnpike}

Our closed-loop results rely on a stochastic notion of dissipativity and the so-called turnpike property.
These concepts, together with their relation to MPC and receding horizon control, are well established in deterministic settings, cf. \cite{Faulwasser2018, Faulwasser2022, Gruene2022, Gruene2013, Gruene2016, Breiten2020, Angeli2012, Mueller2014, Koehler2018}.
More recently, dissipativity and turnpike properties have been extended to stochastic systems \cite{Sun2022, Sun2023, Gros2022, Schiessl2023a, Schiessl2024a, Schiessl2025}.
In order to formulate these notions in a stochastic framework, an appropriate analogue of the deterministic steady state is required.
To this end, we recall the following definition of a stationary process.

\begin{defi}[Stationary stochastic processes] \label{defn:stationaryProcess}
    A pair of stochastic processes $(\mathbf{X}^s,\mathbf{U}^s)$ given by
    \begin{equation} \label{eq:sys_stat}
        X^s(\tol+1) = f(X^s(\tol), U^s(\tol), W(\tol))
    \end{equation}
    with $\mathbf{U}^s \in \UU_{ad}^{\infty}(0,X^s(0))$ is called stationary for system \eqref{eq:stochSysTimeVarying} if there exist probability distributions $P^s_X$, $P^s_U$, and $P^s_{X,U}$ with
    \begin{equation*}
    \begin{split}
        X^s(\tol) \sim P^s_X, \quad U(\tol) \sim P^s_U, \quad (X^s(\tol),U^s(\tol)) \sim P^s_{X,U}
    \end{split}
    \end{equation*}
    for all $\tol \in \N_0$. 
\end{defi}

\begin{rem}
    \begin{enumerate}[(i)]
        \item The existences of a stationary stochastic process is equivalent to the existence of a stationary distribution-policy pair. Thus, the infinite horizon characterization from Definition~\ref{defn:stationaryProcess} can be reformulated as an equilibrium problem on the space of probability distributions and Markov policies. For more details on this we refer to \cite[Section~II.B]{Schiessl2024a}.
        \item Since the joint distribution of the stationary pair $(\mathbf{X}^s, \mathbf{U}^s)$ remains invariant over time, the corresponding stage cost is also constant at all times.
        Accordingly, we denote this constant value by $\ell(\mathbf{X}^s, \mathbf{U}^s)$ from here on.
    \end{enumerate}
\end{rem}

Now we can use the stationary pair of stochastic processes from Definition~\ref{defn:stationaryProcess} to introduce our notion of stochastic dissipativity. To this end, we make use of a (pseudo)metric $d: \RR{\Omega,\X} \times \RR{\Omega,\X} \to [0,\infty]$. 
Throughout the paper we assume that $d$ is well defined for all random variables in $\RR{\Omega,\X}$, though it may take the value $+\infty$. 

Furthermore, we adopt the standard classes of comparison functions, cf.\ \cite{Khalil2002}:
\begin{align*}
    \K &:= \{ \alpha: \R_0^+ \to \R_0^+ \mid \alpha \text{ is continuous and strictly increasing, with } \alpha(0)=0 \}, \\
    \K_{\infty} &:= \{ \alpha: \R_0^+ \to \R_0^+ \mid \alpha \in \K \text{ and } \alpha \text{ is unbounded} \}, \\
    \mathcal{L} &:= \{ \delta: \R_0^+ \to \R_0^+ \mid \delta \text{ is continuous and strictly decreasing, with } \lim_{t \to \infty}\delta(t)=0 \}, \\
    \K\mathcal{L} &:= \{ \beta: \R_0^+ \times \R_0^+ \to \R_0^+ \mid \beta \text{ is continuous, } \beta(\cdot,t) \in \K~\forall t \in \R_0^+, \ \beta(r,\cdot) \in \mathcal{L}~\forall r \in \R_0^+ \}.
\end{align*}

\begin{defi}[Stochastic dissipativity] \label{defn:stochDissi}
    Consider a pair of stationary stochastic processes $(\mathbf{X}^s,\mathbf{U}^s)$ according to Definition~\ref{defn:stationaryProcess} with $\vert \ell(\mathbf{X}^s,\mathbf{U}^s) \vert < \infty$ and a (pseudo)metric $d$ on the space of random variables.
    Then, we call the stochastic optimal control problem \eqref{eq:stochOCPopenloop} strictly stochastically dissipative at $(\mathbf{X}^s,\mathbf{U}^s)$ with respect to the (pseudo)metric $d$, if there exists a storage function $\lambda: \N_0 \times \RR{\Omega,\X} \rightarrow \R$ uniformly bounded from below and a function $\alpha \in \K_{\infty}$ 
    such that 
    \begin{equation} \label{eq:DissiIneq}
    \begin{split}
        \tilde{\ell}(\tcl,X,U) := \ell(&X,U) - \ell(\mathbf{X}^s,\mathbf{U}^s)
        + \lambda(\tcl,X) - \lambda(\tcl+1,f(X,U,W(\tcl))
        \geq \alpha( d(X,X^s(\tcl)) )
    \end{split}
    \end{equation}
    holds for all $\tcl \in \N_0$ and all $U \in \UU_{ad}(X)$. The system is called dissipative if inequality \eqref{eq:DissiIneq} holds with $\alpha \equiv 0$.
\end{defi}

Originally, dissipativity was introduced by Willems in \cite{Willems1972part1,Willems1972part2} with an energy-based interpretation, stating that a system cannot generate energy internally but only store energy supplied from the outside, while strict dissipativity additionally requires that a certain amount of stored energy is irreversibly dissipated to the environment.
Indeed, Willems' ambition was to extend Lyapunov concepts to systems with inputs and outputs:
``A generalization of Lyapunov functions to open systems, to systems with inputs and outputs.'' \cite{Willems2007}.
In optimal control, dissipativity links the running cost to a storage function that captures the ability of a system to accumulate and release energy over time. In the strict case, it further penalizes persistent deviations from an optimal steady state through an additional dissipation term. 
This interpretation naturally extends to the stochastic setting by replacing the deterministic steady state with a stationary stochastic process, reflecting that optimal operation is characterized not by a single state but by invariant statistical behavior.
Moreover, the metric or pseudometric employed in the dissipation term specifies which aspects of the deviation from the optimal stationary process are considered relevant, such as moment discrepancies, distributional distances, or pathwise fluctuations, leading to stronger or weaker implications, as we see later.

The next definition introduces stochastic dissipativity notions with respect to specific metrics that are of particular interest in this paper and have already been partially introduced in \cite{Schiessl2025}.

\begin{defi}[Types of stochastic dissipativity] \label{defn:stochDissiTypes}
    Assume that the optimal control problem is (strictly) stochastically dissipative according to Definition \ref{defn:stochDissi}. Then we say that it is (strictly)
    \begin{enumerate}[(i)]
        \item \emph{$L^p$ dissipative} if $d$ is the $L^p$-norm, i.e.\ 
        \begin{equation} \label{eq:LpNorm}
            d(X,Y) = \Vert X - Y \Vert_{L^p} := \Exp{\Vert X - Y \Vert^p}^{\frac{1}{p}}.
        \end{equation}
        \item \emph{pathwise-in-probability dissipative} if $d$ is the Ky-Fan metric, i.e.\ 
        \begin{equation} \label{eq:KyFanMetric}
            d(X,Y) = d_{KF}(X,Y) := \inf_{\eps > 0} \{\Prob(\Vert X - Y \Vert > \eps) \leq \eps \}.
        \end{equation}
        \item \emph{distributionally dissipative} if $d$ is a metric on the space of probability measures $\mathcal{M}(\X)$, or more precisely
        \begin{enumerate}[(a)]
            \item \emph{Wasserstein dissipative of order $p$} if $d$ is the Wasserstein distance of order $p$, i.e\ 
            \begin{equation} \label{eq:WassersteinMetric}
                d(X,Y) = \inf \{ \Vert \bar{X} - \bar{Y} \Vert_{L^p} \mid \bar{X} \sim X, \bar{Y} \sim Y \}.
            \end{equation}
            \item \emph{weakly distributionally dissipative} if $d$ is the Lévy-Prokhorov metric, i.e.\ 
            \begin{equation} \label{eq:LevyProhkorovMetric}
                d(X,Y) = \inf \{ d_{KF}(\bar{X}, \bar{Y}) \mid \bar{X} \sim X, \bar{Y} \sim Y \}.
            \end{equation}
        \end{enumerate}
        \item \emph{$p$-th moment dissipative} if 
        \begin{equation} \label{eq:MomentsMetric}
            d(X,Y) = \left\vert \Exp{\Vert X \Vert^p}^{\frac{1}{p}} - \Exp{\Vert Y \Vert^p}^{\frac{1}{p}} \right\vert.
        \end{equation}
    \end{enumerate}
\end{defi}

\begin{rem}
    Note that the well-definedness assumption preceding Definition~\ref{defn:stochDissi} is satisfied for all metrics introduced in Definition~\ref{defn:stochDissiTypes}(i)-(iii). This is because probabilities always lie in $[0,1]$, and for nonnegative random variables $Z$ the expected value $\E[Z]$ always exists in $\R_{\geq 0} \cup \{\infty\}$. 
    The metric form (iv) satisfies the well-definedness assumption if at least one of the random variables has a finite $p$-th moment.
\end{rem}
As examples of stochastically dissipative optimal control problems, we refer to \cite{Schiessl2025}, where it was shown that the class of generalized linear-quadratic stochastic optimal control problems possesses these properties, and to \cite{Schiessl2024a} for a nonlinear example.

In addition to dissipativity, the turnpike property will play an important role in our analysis.
The turnpike property states that optimal trajectories are, for most of the time, close to a stationary stochastic process.
This means that although the optimal solutions of the problem \eqref{eq:stochOCP} are not stationary, they are approximately stationary except for a number of points bounded independently of the horizon $N$.
This property is formalized in the following definition.
Here, $\#\mathcal{Q}$ denotes the number of elements of the set $\mathcal{Q}$.

\begin{defi}[Stochastic turnpike properties] \label{defn:stochTurnpike}
    Consider a stationary pair $(\mathbf{X}^s,\mathbf{U}^s)$ with $\vert \ell(\mathbf{X}^s,\mathbf{U}^s) \vert < \infty$ and a (pseudo)metric $d$ on the space $\RR{\Omega,\X}$ and define the closed ball around $X^s(\tcl)$ with respect to the (pseudo)metric $d$ and radius $r>0$ as
    \begin{equation}
        \Bcl^d_r(X^s(\tcl)) := \{X \in \RR{\Omega,\X} \mid d(X,X^s(\tcl)) \leq r \}.
    \end{equation}   
    Then, we say that the optimal control problem \eqref{eq:stochOCPopenloop} has the uniform stochastic turnpike property locally around $\mathbf{X}^s$ on 
    \begin{enumerate}[(i)]
        \item finite horizon if there exists $r >0$ and $\vartheta \in \mathcal{L}$ such that for each $\tcl \in \N_0$, each optimal trajectory $X_{\mathbf{U}^*_{N}}(\cdot;\tcl,X_\tcl)$ with $X_\tcl \in \Bcl_r^d(X^s(\tcl))$ and all $N,L \in \N$ there is a set $\mathcal{Q}(\tcl,X_\tcl,L,N) \subseteq \{0,\ldots,N\}$ with $\# \mathcal{Q}(\tcl,X_\tcl,L,N) \leq L$ elements and 
        \begin{equation*}
            d \left( X_{\mathbf{U}^*_{N}}(\tol;\tcl,X_\tcl) , X^s(\tol+\tcl) \right) \leq \vartheta(L)
        \end{equation*}
        for all $\tol \in \{0,\ldots,N\} \setminus \mathcal{Q}(\tcl,X_\tcl,L,N)$;
        \item infinite horizon if there exists $r >0$ and $\vartheta_{\infty} \in \mathcal{L}$ such that for each $\tcl \in \N_0$, each optimal trajectory $X_{\mathbf{U}^*_{\infty}}(\cdot;\tcl,X_\tcl)$ with $X_\tcl \in \Bcl_r^d(X^s(\tcl))$ and all $L \in \N$ there is a set $\mathcal{Q}(\tcl,X_\tcl,L,\infty) \subseteq \N_0$ with $\# \mathcal{Q}(\tcl,X_\tcl,L,\infty) \leq L$ elements and 
        \begin{equation*}
            d \left( X_{\mathbf{U}^*_{\infty}}(\tol;\tcl,X_\tcl) , X^s(k+\tol) \right) \leq \vartheta_{\infty}(L)
        \end{equation*}
        for all $\tol \in \N_0 \setminus \mathcal{Q}(\tcl,X_\tcl,L,\infty)$.
    \end{enumerate}
\end{defi}

While the term \emph{uniform} in Definition~\ref{defn:stochTurnpike} refers to the fact that the function $\vartheta$ does not depend on the initial time $\tcl$ and state $X_\tcl$, the phrase \emph{locally} indicates that the property must only hold for trajectories that are sufficiently near to the stationary process at time $\tcl$.

Again, different choices of the metric in Definition~\ref{defn:stochTurnpike} will lead to stronger or weaker properties.
The next definition introduces special turnpike properties with respect to specific metrics analogously to Definition~\ref{defn:stochTurnpikeTypes}.

\begin{defi}[Types of stochastic turnpike properties] \label{defn:stochTurnpikeTypes}
    Assume that the optimal control problem has the stochastic turnpike property according to definition \ref{defn:stochTurnpike}. Then we say that it has 
    \begin{enumerate}[(i)]
        \item the \emph{$L^p$ turnpike property} if $d$ is the $L^p$-norm from equation~\eqref{eq:LpNorm}.
        \item the \emph{pathwise-in-probability turnpike property} if $d$ is the Ky-Fan metric from equation \eqref{eq:KyFanMetric}.
        \item the \emph{distributional turnpike property} if $d$ is a metric on the space of probability measures $\mathcal{M}(\X)$, or more precisely
        \begin{enumerate}[(a)]
            \item the \emph{Wasserstein turnpike property of order $p$} if $d$ is the Wasserstein distance of order $p$ from equation~\eqref{eq:WassersteinMetric}.
            \item the \emph{weak distributional turnpike property} if $d$ is the Lévy-Prokhorov metric from equation~\eqref{eq:LevyProhkorovMetric}.
        \end{enumerate}
        \item the \emph{$p$-th moment turnpike property} if $d$ is the metric from equation~\eqref{eq:MomentsMetric}.
    \end{enumerate}
\end{defi}

There is a deep connection between (strict) dissipativity and turnpike properties. For example, it has been shown in~\cite{Schiessl2025} that strict dissipativity implies the turnpike property for near-steady-state solutions in stochastic settings.
In this paper, however, we are more concerned with optimal solutions rather than near-steady-state behavior. As shown below, this implication also holds for optimal solutions under slightly different conditions.
To this end, we introduce the following definition of cheap reachability, describes the property that the system can be steered toward an optimal stationary process at a bounded additional cost relative to the cost required to maintain the state in the stationary regime, independently of the horizon $N$.

\begin{defi}[Cheap reachability]
    \label{defn:cheapReachability}
    Consider a stationary pair $(\mathbf{X}^s,\mathbf{U}^s)$ and a (pseudo)\-metric $d$ on the space $\RR{\Omega,\X}$.
    Then, the stationary process $\mathbf{X}^s$ is called cheaply reachable on 
    $$\Bcl_r^d(\mathbf{X}^s) := \{ (\tcl,X_\tcl) \mid \tcl \in \N_0, ~ X_\tcl \in \Bcl_r^d(X^s(\tcl)) \},$$
    if there exists $r,C_V > 0$ such that for each $(\tcl,X_\tcl) \in \Bcl_r^d(\mathbf{X}^s)$ and all $N \in \N$ the inequality $V_N(\tcl,X_\tcl) - N \ell(\mathbf{X}^s,\mathbf{U}^s) \leq C_V$ holds.
\end{defi}

We can now show that strict dissipativity also implies the turnpike property for optimal solutions, assuming cheap reachability and an additional boundedness property of the storage function.

\begin{thm} \label{thm:Turnpike} 
    Assume that the optimal control problem~\eqref{eq:stochOCPopenloop} is strictly stochastically dissipative at $(\mathbf{X}^s,\mathbf{U}^s)$ with respect to some (pseudo)metric $d$.
    Further, assume that $\mathbf{X}^s$ is cheaply reachable on $\Bcl^d_r(\mathbf{X}^s)$ and that the storage function $\lambda$ is uniformly bounded on $\Bcl^d_r(\mathbf{X}^s)$.
    Then, the stochastic optimal control problem~\eqref{eq:stochOCPopenloop} has the uniform stochastic finite-horizon turnpike property on $\Bcl^d_r(\mathbf{X}^s)$ with respect to the (pseudo)metric $d$.
\end{thm}
\begin{proof}
    By strict dissipativity we know that there are $\alpha \in \K_{\infty}$, $C^l_{\lambda} > 0$ and $\lambda$ with $\lambda(\tcl,X_\tcl) > -C^l_{\lambda}$ for all $\tcl \in \N_0$ and $X_\tcl \in \RR{\Omega,\X}$ such that
    \begin{equation}
    \begin{split} \label{eq:lowerBoundCost}
        V_N(&\tcl,X_\tcl) - N\ell(\mathbf{X}^s,\mathbf{U}^s) \\
        &=
        \sum_{\tol=0}^{N-1} \left( \ell(X_{\mathbf{U}_N^*}(\tol;\tcl,X_\tcl), U(\tol)) - \ell(\mathbf{X}^s,\mathbf{U}^s) \right) \\
        \geq& \sum_{\tol=0}^{N-1} \big(\alpha(d(X_{\mathbf{U}_N^*}(\tol;\tcl,X_\tcl),X^s(\tcl+\tol)))
        - \lambda(\tcl+\tol,X_{\mathbf{U}_N^*}(\tol;\tcl,X_\tcl)) \\
        &~~~~~~+ \lambda(\tcl+\tol+1,X_{\mathbf{U}_N^*}(\tol+1;\tcl,X_\tcl)) \big) \\
        =&  - \lambda(\tcl,X_{\mathbf{U}_N^*}(0;\tcl,X_\tcl)) + \lambda(\tcl+N,X_{\mathbf{U}_N^*}(N;\tcl,X_\tcl)) \\
        &+ \sum_{\tol=0}^{N-1} \alpha(d(X_{\mathbf{U}_N^*}(\tol;\tcl,X_\tcl),X^s(\tcl+\tol))).
    \end{split}
    \end{equation}
    Now suppose that the finite-horizon turnpike property does not hold with 
    \begin{equation} \label{eq:vartheta}
        \vartheta(L) := \alpha^{-1}\left( \dfrac{C_\lambda + C^l_{\lambda} + C_V}{L} \right) \in \mathcal{L}
    \end{equation}
    where $C_V > 0$ is from Definition~\ref{defn:cheapReachability} and $C_{\lambda} > 0$ is an uniform bound on $\lambda$, i.e.\ $\vert \lambda(\tcl,X_\tcl) \vert < C_{\lambda}$ for all $X_\tcl \in \Bcl^d_r(X^s(\tcl))$ and all $\tcl \in \N_0$.
    Then, there exists $N \in \N$, $X_\tcl \in \Bcl_r^d(X^s(\tcl))$, $L \in \N$ and a set $\mathcal{M} \subset \lbrace 0,\ldots,N-1 \rbrace$ of at least $L+1$ time instants such that $d(X_{\mathbf{U}_N^*}(\tol;\tcl,X_\tcl),X^s(\tcl+\tol)) > \vartheta(L)$. 
    Using equation~\eqref{eq:lowerBoundCost} this yields
    \begin{equation*}
    \begin{split}
        V_N(\tcl,X_\tcl) - N\ell(\mathbf{X}^s,\mathbf{U}^s) 
        \geq& - \lambda(\tcl,X_{\mathbf{U}_N^*}(0;\tcl,X_\tcl)) + \lambda(\tcl+N,X_{\mathbf{U}_N^*}(N;\tcl,X_\tcl)) \\
        &+ \sum_{\tol \in \mathcal{M}} \alpha(d(X_{\mathbf{U}_N^*}(\tol;\tcl,X_\tcl),X^s(\tcl+\tol))) \\
        >&- C_{\lambda} - C^l_{\lambda} + L\alpha(\vartheta(L)) = C_V,
    \end{split}
    \end{equation*}
    which contradicts the cheap reachability assumption and thus proves the claim.
\end{proof}

\section{Stability of stochastic MPC} \label{sec:Stability}
After introducing the tools that form the basis of our findings, we begin analyzing the closed-loop behavior of the proposed MPC algorithm. In particular, we will demonstrate the stability of the closed-loop trajectory and derive near-optimal performance bounds. 
However, before investigating stability, we introduce the key assumptions underlying our analysis. 
To this end, we start with the following assumption, which states that the open-loop problem is strictly dissipative.

\begin{ass} \label{ass:Dissipativity}
    We assume that the stochastic optimal control problem~\eqref{eq:stochOCPopenloop} is strictly stochastically dissipative at $(\mathbf{X}^s,\mathbf{U}^s)$ with respect to a (pseudo)metric $d$ on $\RR{\Omega,\X}$ according to Definition~\ref{defn:stochDissi}.
\end{ass}

In addition to strict dissipativity, we also want the problem~\eqref{eq:stochOCPopenloop} to enjoy a stochastic turnpike property. For this purpose, we make the following assumption, which ensures that the conditions of Theorem~\ref{thm:Turnpike} are satisfied.

\begin{ass} \label{ass:BoundedStorage_CheapReach}
    Consider the stationary pair $(\mathbf{X}^s,\mathbf{U}^s)$ and the (pseudo)metric $d$ from Assumption~\ref{ass:Dissipativity}. 
    Then, we assume that there is $r > 0$ such that
    \begin{enumerate}[(i)]
        \item the storage function $\lambda: \N_0 \times \RR{\Omega,\X} \to \R$ from Assumption~\ref{ass:Dissipativity} is uniformly bounded on $\Bcl^d_r(\mathbf{X}^s)$.
        \item the stationary process $\mathbf{X}^s$ is finite-horizon cheaply reachable on $\Bcl^d_r(\mathbf{X}^s)$.
    \end{enumerate}
\end{ass}

Note that Assumption~\ref{ass:BoundedStorage_CheapReach}, in conjunction with Assumption~\ref{ass:Dissipativity}, guarantees the turnpike property on finite horizon.
For the purpose of our stability analysis, however, this finite-horizon property is sufficient. Moreover, as we will later demonstrate in Lemma~\ref{lem:finiteV}, the finite-horizon cheap reachability carries over to the infinite-horizon, thereby implying the turnpike property also in the infinite-horizon case.

To conclude our introductory remarks, we make the following additional continuity assumption on the optimal value function.

\begin{ass} \label{ass:ContinuityV}
    Consider the stationary pair $(\mathbf{X}^s,\mathbf{U}^s)$ and the (pseudo)metric $d$ from Assumption~\ref{ass:Dissipativity}. 
    Then, we assume that $V$ is approximately continuous at $\mathbf{X}^s$, i.e.\ there exist $\eps > 0$, $\gamma_V: \R_{\geq 0} \times \R_{\geq 0} \to \R_{\geq 0}$ with $\gamma_V(N,r) \to 0$ if $N \to \infty$ and $r \to 0$, and $\gamma_V(\cdot,r),\gamma_V(N,\cdot)$ monotonous for fixed $r$ and $N$ such that
    $$\vert V_N(\tcl,X_\tcl) - V_N(\tcl,X^s(\tcl)) \vert \leq \gamma_{V}(N,d(X_\tcl,X^s(\tcl)))$$
    holds for all $N \in \N$, $\tcl \in \N_0$ and $X_\tcl \in \Bcl_{\eps}^d(X^s(\tcl))$.
\end{ass}

Now, we will start to investigate the stability properties of our proposed stochastic MPC algorithm.
For this purpose, we introduce the following form of asymptotic stability.

\begin{defi} [Uniform P-practical asymptotic stability] \label{defn:stability}
    Let $\mathcal{Y}(\tcl) \subset \mathcal{X}$, $\tcl \in \mathbb{N}_0$ be a forward invariant family of sets, i.e.\ $f(X_\tcl,\mu_N(X_\tcl),W(\tcl)) \in \mathcal{Y}(\tcl+1)$ for all $X_\tcl \in \mathcal{Y}(\tcl)$, and let $\mathcal{P}(\tcl) \subset \mathcal{Y}(\tcl)$, $\tcl \in \mathbb{N}_0$ be subsets of $\mathcal{Y}(\tcl)$. 
    The stationary process $\mathbf{X}^s$ with $X^s(\tcl) \in \mathcal{Y}(\tcl)$ for all $\tcl \in \mathbb{N}_0$ is called P-practically uniformly asymptotically stable on $\mathcal{Y}(\tcl)$ with respect to the (pseudo)metric $d$ if there exists a comparison function $\beta \in \mathcal{K}\mathcal{L}$ such that
    \begin{equation} \label{eq:stability}
        d(X_{\mu_N}(\tcl;\tcl_0,X_{\tcl_0}),X^s(\tcl+\tcl_0)) \leq \beta(d(X,X^s(\tcl_0)),\tcl)
    \end{equation}
    is valid for all $X_{\tcl_0} \in \mathcal{Y}(\tcl_0)$ and all $\tcl_0,\tcl \in \mathbb{N}_0$ with $\tcl \geq \tcl_0$ and $X_{\mu_N}(\tcl;\tcl_0,X_{\tcl_0}) \notin \mathcal{P}(\tcl)$. 
\end{defi}

Since the decay condition \eqref{eq:stability} only has to hold outside of the set $\mathcal{P}(j)$, unlike to asymptotic stability, which requires that the system converges exactly to the stationary problem as time goes to infinity, P-practical asymptotic stability only guarantees that the state converges to a neighborhood of the stationary process.
Typically this neighborhood is chosen to be small (as in Corollary \ref{cor:semiprac}, below).
In order to establish this form of stability, one usually does not directly verify its definition, but rather uses the concept of Lyapunov functions.
A Lyapunov function in our time-varying setting is defined in the following way.

\begin{defi}[Uniform time-varying Lyapunov function] \label{defn:lyapunovFunction}
    Consider $\mathcal{S}(\tcl) \subseteq \RR{\Omega,\X}$ for $\tcl \in \mathbb{N}_0$ and define $\mathcal{S} := \{(\tcl,X_\tcl)| \tcl \in \mathbb{N}_0, X_\tcl \in \mathcal{S}(\tcl)\}$. 
    A function $V: \mathcal{S} \rightarrow \mathbb{R}_0^+$ is called a uniform time-varying Lyapunov function on $\mathcal{S}(\tcl)$ if there are $\alpha_1(\cdot), \alpha_2(\cdot) \in \mathcal{K}_{\infty}$ and $\alpha_V \in \mathcal{K}$ such that
    \begin{enumerate}[(i)]
        \item $\alpha_1(d(X_\tcl,X^s(\tcl))) \leq V(\tcl,X_\tcl) \leq \alpha_2(d(X_\tcl,X^s(\tcl)))$
        \item  $V(\tcl+1,f(X_\tcl,\mu_N(X_\tcl),W(\tcl))) \leq V(\tcl,X_\tcl) - \alpha_V(d(X_\tcl,X^s(\tcl)))$
    \end{enumerate}
    holds for all $\tcl \in \mathbb{N}_0$ and all $X_\tcl \in \mathcal{S}(\tcl)$ with $f(X_\tcl,\mu_N(X_\tcl),W(\tcl)) \in S(\tcl+1)$.
\end{defi}

The time-varying Lyapunov function from Definition~\ref{defn:lyapunovFunction} can be used to establish practical asymptotic stability as follows.

\begin{thm}[{cf.\ \cite[Theorem~2.23]{Gruene2017b}}] \label{thm:stabilityLyapunov}
    Consider forward invariant families of sets $\mathcal{Y}(\tcl)$ and $\mathcal{P}(\tcl) \subset \mathcal{Y}(\tcl)$ with $X^s(\tcl) \in P(\tcl)$ $\tcl \in \mathbb{N}_0$. 
    If there exists a uniform time-varying Lyapunov function $V$ on $\mathcal{S}(\tcl) = \mathcal{Y}(\tcl) \setminus \mathcal{P}(\tcl)$, then $\mathbf{X}^s$ is uniformly $P$-practically asymptotically stable on $\mathcal{Y}(\tcl)$.
\end{thm}

Our goal in this section is now to show that the optimal value function is a time-varying Lyapunov function for our closed-loop systems on appropriate sets.
However, the optimal value function of our original problem~\eqref{eq:stochOCPopenloop} is not suitable for this.
Instead, we will consider the following modified problem.

\begin{defi}[Modified optimal control problem] \label{defn:modifiedOCP}
    Assume that the optimal control problem \eqref{eq:stochOCPopenloop} is stochastic dissipative at $(\mathbf{X}^s,\mathbf{U}^s)$.
    Then we define the modified optimal control problem as 
    \begin{equation}
    \label{eq:stochOCPmodified}
        \begin{split}
            \minimize_{\mathbf{U} \in \UU_{ad}^{N}(\tcl,X_\tcl)} &~\tilde{J}_N(\tcl,X_\tcl,\mathbf{U}) := \sum_{\tol=0}^{N-1} \tilde{\ell}(\tol,X(\tol;\tcl,X_\tcl),U(\tol))
        \end{split}
    \end{equation}
    with $\tilde{\ell}(\tol,X,U)$ from equation~\eqref{eq:DissiIneq}.
    Moreover, for all $N \in \N$ we define the corresponding optimal value function as $\tilde{V}_N(\tcl,X_\tcl) := \inf_{\mathbf{U} \in \UU_{ad}^{N}(\tcl,X_\tcl)} \tilde{J}_N(\tcl,X_\tcl,\mathbf{U})$ and if a minimizer of \eqref{eq:stochOCPmodified} exists, we denote it by $\tilde{\mathbf{U}}^*_{N,\tcl,X_\tcl}$ (or short $\tilde{\mathbf{U}}^*_{N}$).
\end{defi}

In contrast to the original problem in Definition~\ref{defn:modifiedOCP}, we use the stage costs from equation~\eqref{eq:DissiIneq}, which are modified by the storage function $\lambda$.
Due to this, the costs $\tilde{J}_N(\tcl,X_\tcl,\mathbf{U})$ are now also time-dependent on their own and not only indirectly through the dynamics as before.
Moreover, this will also lead to different optimal trajectories, i.e., in general it holds that $\tilde{\mathbf{U}}^*_N \neq \mathbf{U}^*_N$.
However, the first question for us will now be whether, despite this, the general properties of the stochastic optimal control problem~\eqref{eq:stochOCPopenloop}, such as the turnpike property and dissipativity, are transferred to the modified problem.
While it is clear that dissipativity will hold for the modified problem, since $\tilde{\ell}(\tcl,X,U) \geq \alpha(d(X,X^s(\tcl)))$ holds without any further modification, it is not clear that this also applies to the turnpike property - in particular, since we do not know whether the modified problem is also cheaply reachable.
However, under the following additional assumption, we can show that this is also true.

\begin{ass} \label{ass:modifiedStagecost}
    For the stationary pair $(\mathbf{X}^s,\mathbf{U}^s)$ from Assumption~\ref{ass:Dissipativity} it holds that $\tilde{\ell}(\tcl,X^s(\tcl),U^s(\tcl)) = 0$ for all $\tcl \in \N_0$.
\end{ass}

Assumption~\ref{ass:modifiedStagecost} states that the modified stage costs evaluated at the stationary pair are not only bounded from below by zero, but are in fact equal to zero.
It should be noted that this is equivalent to the requirement that the storage function is constant along the optimal trajectory, i.e.\ $\lambda(\tcl,X^s(\tcl)) = \lambda(\tcl+1,X^s(\tcl+1))$ for all $\tcl \in \N_0$, since the stage costs $\ell(\mathbf{X}^s,\mathbf{U}^s)$ remain constant along the stationary pair.
While for time-invariant problems it is clear that this condition must always hold, since then $X^s(\tcl) = X^s(\tcl+1)$ holds for all $\tcl \in \mathbb{N}_0$, for time-dependent problems this does not necessarily hold by definition. 
However, for all problems for which dissipativity has been theoretically guaranteed by us so far, cf.\ \cite{Schiessl2025,Schiessl2024a}, this property holds, which is why we do not consider this assumption restrictive.
Now we are able to guarantee the turnpike property for the modified problem as shown in the next lemma.

\begin{lem} \label{lem:TurnpikeModified}
    Let Assumption~\ref{ass:Dissipativity}, Assumption~\ref{ass:BoundedStorage_CheapReach}, and Assumption~\ref{ass:modifiedStagecost} hold and consider the stationary pair $(\mathbf{X}^s,\mathbf{U}^s)$ and the (pseudo)metric $d$ from Assumption~\ref{ass:Dissipativity}.
    Then, the modified stochastic optimal control problem~\eqref{eq:stochOCPmodified} 
    has the uniform stochastic finite-horizon turnpike property locally on $\Bcl_r^d(\mathbf{X}^s)$ with respect to the (pseudo)metric $d$.
\end{lem}
\begin{proof}
    Consider $\tcl \in \N_0$, $X_\tcl \in \Bcl_r^d(X^s(\tcl))$ and let $\mathbf{U}^*_N$ be the optimal solution of the stochastic optimal control problem~\eqref{eq:stochOCPopenloop} with initial time $\tcl$ and state $X_\tcl$. 
    Then, by optimality it follows that
    \begin{equation*}
    \begin{split}
        \tilde{V}_N(\tcl,X_\tcl) \leq& \tilde{J}_N(\tcl,X_\tcl,\mathbf{U}_N^*) \\
        =& \lambda(\tcl,X_\tcl) + J_N(X_\tcl,\mathbf{U}_N^*) - N \ell(\mathbf{X}^s,\mathbf{U}^s) - \lambda(\tcl+N,X_{\mathbf{U}^*_N}(N;\tcl,X_\tcl)).
    \end{split}
    \end{equation*}
    Using Assumption~\ref{ass:BoundedStorage_CheapReach}(ii) we get
    \begin{equation*}
       J_N(X_\tcl,\mathbf{U}_N^*) - N \ell(\mathbf{X}^s,\mathbf{U}^s) = V_N(\tcl,X_\tcl) - N \ell(\mathbf{X}^s,\mathbf{U}^s) \leq C_V.
    \end{equation*}
    Furthermore, by Assumption~\ref{ass:BoundedStorage_CheapReach}(i) we know that $\lambda(\tcl,X_\tcl) < C_{\lambda}$ for some $C_{\lambda} > 0$ and due to the definition of dissipativity we obtain that there is a $C^l_{\lambda} > 0$ such that $\lambda(\tcl,X) > - C^l_{\lambda}$ for all $\tcl \in \N_0$ and $X \in \RR{\Omega,\X}$.
    Hence, we get
    \begin{equation*}
        \tilde{V}_N(\tcl,X_\tcl) \leq \tilde{C}_V := C_{\lambda} + C_V + C^l_{\lambda}
    \end{equation*}
    for all $\tcl \in \N_0$ and $X_\tcl \in \Bcl_r^d(X^s(\tcl))$.
    Thus, the proof follows analogously to the one of Lemma~\ref{thm:Turnpike} with
    \begin{equation} \label{eq:varthetaModified}
        \vartheta(L) := \alpha^{-1}\left( \dfrac{\tilde{C}_V}{L} \right) = \alpha^{-1}\left( \dfrac{C_{\lambda} + C_V + C^l_{\lambda}}{L} \right) \in \mathcal{L}
    \end{equation}
    since $\tilde{\ell}(\tcl,X^s(\tcl),U^s(\tcl)) = 0$ for all $\tcl \in \N_0$ due to Assumption~\ref{ass:modifiedStagecost} and the fact that the modified problem~\eqref{eq:stochOCPmodified} is always dissipative with storage function $\lambda \equiv 0$.
\end{proof}

Note that the $\mathcal{L}$-function $\vartheta$ from equation~\eqref{eq:varthetaModified} for the turnpike property of the modified problem and the one from equation~\eqref{eq:vartheta} for the original problem are the same.
Thus, we do not have to distinguish between two different functions $\vartheta$ and $\tilde{\vartheta}$ in the following when we use turnpike properties in our estimates.
However, it should be noted that, in contrast to this, the proofs do not show that the sets of exceptional points $\mathcal{Q}(\tcl,X_\tcl,L,N)$ are the same for the two problems.

Before proceeding, similar to Definition~\ref{ass:ContinuityV}, we impose local continuity assumptions on the modified problem, as stated below.
Note that in contrast to Assumption~\ref{ass:ContinuityV}, the continuity of the modified value function must be satisfied uniformly in $N$.

\begin{ass} \label{ass:ContinuityModified}
    Consider the stationary pair $(\mathbf{X}^s,\mathbf{U}^s)$ and the (pseudo)me\-tric $d$ from Assumption~\ref{ass:Dissipativity}. Then, there exist $\tilde{r} \geq \tilde{\eps} > 0$ and $\gamma_{\lambda},\gamma_{\tilde{V}} \in \K_{\infty}$ such that for all  $N \in \N$, $\tcl \in \N_0$ it holds that 
    \begin{enumerate}[(i)]
        \item $\vert \lambda(\tcl,X_\tcl) - \lambda(\tcl,X^s(\tcl)) \vert \leq \gamma_{\lambda}(d(X_\tcl,X^s(\tcl)))$ for all $X_\tcl \in \Bcl_{\tilde{\eps}}^d(X^s(\tcl))$.
        \item $\vert \tilde{V}_N(\tcl,X_\tcl) - \tilde{V}_N(\tcl,X^s(\tcl)) \vert \leq \gamma_{\tilde{V}}(d(X_\tcl,X^s(\tcl)))$ for all $X_\tcl \in \Bcl_{\tilde{r}}^d(X^s(\tcl))$.
    \end{enumerate}
\end{ass}

We now make a series of estimates necessary to show the practical asymptotic stability of our closed-loop system, starting with the following lemma, which compares the performance up to time $M$ of optimal trajectories on horizons $N$ and $N+1$ of the modified problem.
Note that in the following we have moved all proofs of preparatory lemmas, which can be proved analogously to the deterministic case with slight modifications, to the appendix.

\begin{lem} \label{lem:lem1}
    Let the Assumptions~\ref{ass:Dissipativity}, \ref{ass:BoundedStorage_CheapReach}, and \ref{ass:ContinuityModified}(ii) hold. Consider $\vartheta \in \mathcal{L}$ from \eqref{eq:vartheta} and $\tilde{r}$ from Assumption~\ref{ass:ContinuityModified}.
    Then, for all $\tcl \in \mathbb{N}_0$, $N,L \in \mathbb{N}$ with $L \geq \vartheta^{-1}(\tilde{r})$,
    and all $X_\tcl \in \Bcl^d_{r}(X^s(\tcl))$ there exists a set $\bar{\mathcal{Q}}(\tcl,X_\tcl,L,N)$ with $\# \bar{\mathcal{Q}}(\tcl,X_\tcl,L,N) \leq 2L$ such that for all  $M \in \{0,\ldots,N\} \setminus \bar{\mathcal{Q}}(\tcl,X_\tcl,L,N)$ the estimate
    \begin{equation} \label{eq:lemma1}
        \tilde{J}_M(\tcl,X_\tcl,\tilde{\mathbf{U}}_N^*) = \tilde{J}_M(\tcl,X_\tcl,\tilde{\mathbf{U}}_{N+1}^*) + R_1(\tcl,X_\tcl,M,N)
    \end{equation}
    holds with $\vert R_1(\tcl,X_\tcl,M,N) \vert \leq 2 \gamma_{\tilde{V}}(\vartheta(L))$.
\end{lem}
\begin{proof}
The proof follows exactly the same steps as in the time-varying deterministic case, see \cite[Lemma~1]{Gruene2019}. It is given in  Appendix~\ref{app}.
\end{proof}

Using the above results we can now show that the optimal value function for the modified problem attains nearly identical values for horizons $N$ and $N+1$.

\begin{lem} \label{lem:lem2}
    Let the Assumptions \ref{ass:Dissipativity}, \ref{ass:BoundedStorage_CheapReach}, \ref{ass:modifiedStagecost}, and \ref{ass:ContinuityModified}(ii) hold.
    Consider $\vartheta \in \mathcal{L}$ from \eqref{eq:vartheta}, $\tilde{r}$ from Assumption~\ref{ass:ContinuityModified}, and $\bar{\mathcal{Q}}(\tcl,X_\tcl,L,N)$ from Lemma~\ref{lem:lem1}.
    Then, for all $\tcl \in \mathbb{N}_0$, $N,L \in \mathbb{N}$ with $L \geq \vartheta^{-1}(\tilde{r})$, all $X_\tcl \in \Bcl^d_{r}(X^s(\tcl))$ and all $M \in \{0,\ldots,N\} \setminus \bar{\mathcal{Q}}(\tcl,X_\tcl,L,N)$ the estimate
    \begin{equation} \label{eq:lem2} 
        \tilde{V}_{N+1}(\tcl,X_\tcl) = \tilde{V}_{N}(\tcl,X_\tcl) + R_2(\tcl,X_\tcl,M,N)
    \end{equation}
    holds with $\vert R_2(\tcl,X_\tcl,M,N) \vert \leq 4 \gamma_{\tilde{V}}(\vartheta(L))$ and $\mathcal{M}(\tcl,X_\tcl,L,N)$ from Lemma~\ref{lem:lem1}.
\end{lem}
\begin{proof}
    The proof follows exactly the same steps as in the time-varying deterministic case, see \cite[Lemma~2]{Gruene2019}. It is given in  Appendix~\ref{app}.
\end{proof}

The next lemma considers an initial piece of an optimal trajectory. If this initial piece ends near the turnpike, i.e.\ near the stationary stochastic process $\mathbf{X}^s$, then it has a cost approximately lower than all other trajectories ending there. 

\begin{lem}
\label{lem:theorem1}
    Let Assumptions~\ref{ass:Dissipativity}, \ref{ass:BoundedStorage_CheapReach}, and \ref{ass:ContinuityV} hold. 
    Consider $\vartheta \in \mathcal{L}$ from \eqref{eq:vartheta} and $\eps$ from Assumption~\ref{ass:ContinuityV}.
    Then, for all $\tcl \in \mathbb{N}_0$, $N,L \in \mathbb{N}$ with $L \geq \vartheta^{-1}(\eps)$, and all $X_\tcl \in \Bcl^d_{r}(X^s(\tcl))$ there exists a set $\mathcal{Q}(\tcl,X_\tcl,L,N)$ with $\# \mathcal{Q}(\tcl,X_\tcl,L,N) \leq L$ such that for all $M \in \{0,\ldots,N\} \setminus \mathcal{Q}(\tcl,X_\tcl,L,N)$, all $\mathbf{U} \in \UU_{ad}^M(\tcl,X_\tcl)$ with $d(X_{\mathbf{U}}(M;\tcl,X_\tcl),X^s(\tcl+M)) \leq \vartheta(L)$, the estimate
    \begin{equation} \label{eq:thm1}
        J_M(\tcl,X_\tcl,\mathbf{U}_{N}^*) \leq J_M(\tcl,X_\tcl,\mathbf{U}) + R_3(\tcl,X_\tcl,M,N)
    \end{equation}
    holds with $\vert R_3(\tcl,X_\tcl,M,N) \vert \leq 2 \gamma_{V}(N-M,\vartheta(L))$.
\end{lem}
\begin{proof}
    The proof follows exactly the same steps as in the time-varying deterministic case, see \cite[Lemma~3]{Gruene2019}. It is given in  Appendix~\ref{app}.
\end{proof}

In our final preliminary lemma of the stability analysis, we consider a control sequence $\hat{\mathbf{U}}$ that for the first part consists of the optimal control sequence $\mathbf{U}^*_N$ of the unmodified problem until it is close to the stationary state $\mathbf{X}^s$. 
Then, from the final point, we use the optimal control sequence of the modified problem. 
The lemma states that the resulting composite control sequence has almost the same cost as if we had used the optimal control sequence of the modified problem over the entire horizon.

\begin{lem} \label{lem:lemma3}
    Let Assumptions~\ref{ass:Dissipativity}, \ref{ass:BoundedStorage_CheapReach}, \ref{ass:ContinuityV}, \ref{ass:modifiedStagecost}, and \ref{ass:ContinuityModified}  hold.
    Consider $\vartheta \in \mathcal{L}$ from \eqref{eq:vartheta}, $\eps$ from Assumption~\ref{ass:ContinuityV}, and $\tilde{\eps}$ from Assumption~\ref{ass:ContinuityModified}.
    Let $N,L \in \N$ with $L \geq \vartheta^{-1}(\min \{\eps,\tilde{\eps}\})$, and $\bar{\mathbf{U}}$ be the solution of 
    \begin{equation*}
        \min_{\mathbf{U} \in \UU_{ad}^{N-M}(\tcl+M,\bar{X}_M)}{\tilde{J}_{N-M}(\tcl+M,\bar{X}_M,\mathbf{U})}
    \end{equation*}
    with $\bar{X}_M := X_{\mathbf{U}_N^*}(M;\tcl,X_\tcl)$. \\
    Then, there exists a set $\hat{\mathcal{Q}}(\tcl,X_\tcl,L,N)$ with $\# \hat{\mathcal{Q}}(\tcl,X_\tcl,L,N) \leq 2L$ such that for all $M \in \{0,\ldots,N\} \setminus \hat{\mathcal{Q}}(\tcl,X_\tcl,L,N)$ the composite control sequence $\hat{\mathbf{U}} \in \UU_{ad}^{N}(\tcl,X_\tcl)$ defined by $\hat{U}(\tcl) := U^*_{N}(\tcl)$ for $\tcl = {0,...,M-1}$ and $\hat{U}(\tcl+M) := \bar{U}(\tcl)$ for $\tcl = {0,...,N-M}$ satisfies 
    \begin{equation*}
        \tilde{J}_N(\tcl,X_\tcl,\hat{\mathbf{U}}) = \tilde{V}_N(\tcl,X_\tcl) + R_4(\tcl,X_\tcl,M,N)
    \end{equation*}
    with  
    $|R_4(\tcl,X_\tcl,M,N)| \leq 2\gamma_{\tilde{V}}(\vartheta(L)) + 2\gamma_{\lambda}(\vartheta(L))  + 2\gamma_V(N-M,\vartheta(L))$.
\end{lem}
\begin{proof}
The proof follows exactly the same steps as in the time-varying deterministic case, see \cite[Lemma~3]{Gruene2019}. It is given in Appendix~\ref{app}
\end{proof}

Using all these preparatory results, we can prove that the modified optimal value function $\tilde{V}_N$ is a Lyapunov function for the closed-loop system controlled by the MPC feedback $\mu_N$ and thus is $P$-practically asymptotically stable according to Theorem~\ref{thm:stabilityLyapunov}. Note that the ``practical'' neighborhood $\mathcal{P}(\tcl)$ of the stationary process can be chosen arbitrarily small by increasing the optimization horizon $N$. This will be used in the subsequent Corollary \ref{cor:semiprac}.

\begin{thm} \label{theorem2}
    Let the Assumptions~\ref{ass:Dissipativity}, \ref{ass:BoundedStorage_CheapReach}, \ref{ass:ContinuityV}, \ref{ass:modifiedStagecost}, and \ref{ass:ContinuityModified} hold. 
    Then, there exists $\theta_0 > 0$ and $N_0 \in \N$ such that for all $0 < \theta \leq \theta_0$ the following holds: \\
    There exists a comparison function $\delta \in \mathcal{L}$ such that for all $N \geq N_0$ the optimal value function $\tilde{V}_N$ is a uniform time-varying Lyapunov function for the closed-loop system $f(X,\mu_N(X),W(j))$ on $\mathcal{S}(\tcl) = \mathcal{Y}(\tcl)\backslash \mathcal{P}(\tcl)$ for the families of forward invariant sets $\mathcal{Y}(\tcl) = \tilde{V}_N^{-1}(\tcl,[0,\theta])$ and $\mathcal{P}(\tcl) = \tilde{V}_N^{-1}(\tcl,[0,\delta(N)])$.
\end{thm}
\begin{proof}
    Due to Assumption~\ref{ass:Dissipativity} it follows that 
    \begin{equation}
        \begin{split} \label{eq:lowerBoundLyap}
            \tilde{V}_N(\tcl,X_\tcl) &= \inf_{\mathbf{U} \in \UU_{ad}^N(\tcl,X_\tcl)} \sum_{\tol=0}^{N-1} \tilde{\ell}(\tcl+\tol,X_{\mathbf{U}}(\tol;\tcl,X_\tcl),U(\tol)) \\
            &\geq \inf_{\mathbf{U} \in \UU_{ad}^N(\tcl,X_\tcl)} \sum_{\tol=0}^{N-1}  \alpha(d(X_U(\tol;\tcl,X_\tcl),X^s(\tcl+\tol))) \geq \alpha(d(X_\tcl,X^s(\tcl))) 
        \end{split}
    \end{equation}
    holds for all $X_\tcl \in \RR{\Omega,\X}$ and $\tcl \in \N_0$. 
    Thus, the comparison function $\alpha_1 = \alpha \in \K_{\infty}$ yields a valid lower bound of the optimal value function.
    In the next step we establish an upper bound.
    Since Assumption~\ref{ass:modifiedStagecost} implies that $\tilde{V}_N(\tcl,X^s(\tcl)) = 0$ for all $\tcl \in \N_0$ we obtain by Assumption~\ref{ass:ContinuityModified}(ii) that
    $$\vert \tilde{V}(\tcl,X_\tcl) \vert \leq \gamma_{\tilde{V}}(d(X_\tcl,X^s(\tcl)))$$
    for all $X_\tcl \in \Bcl^d_{\tilde{r}}(X^s(\tcl))$, $\tcl \in \N_0$. 
    Furthermore, by dissipativity we can conclude that for $\theta_1 := \alpha^{-1}(\tilde{r})$ the inequality
    \begin{equation} \label{eq:theta1}
        \theta_1 \geq \tilde{V}_N(X_\tcl) \geq \alpha(d(X_\tcl,X^s(\tcl)))
    \end{equation}
    holds, which yields $\tilde{V}_N^{-1}(\tcl,[0,\theta_1]) \subseteq \Bcl^d_{\tilde{r}}(X^s(\tcl))$ for all $\tcl \in \N_0$.
    Thus, the comparison function $\alpha_2 = \gamma_{\tilde{V}} \in \K_{\infty}$ yields an upper bound of $\tilde{V}_N(\tcl,X_\tcl)$ for all $X_\tcl \in \tilde{V}_N^{-1}(\tcl,[0,\theta_1])$, $\tcl \in \N_0$.
    The next part of the proof shows the decrease of the modified optimal value function along the MPC closed-loop trajectory.
    To this end, consider the composite control sequence $\hat{\mathbf{U}}$ from Lemma \ref{lem:lemma3} and define $X^+_\tcl := X_{\hat{\mathbf{U}}}(1;\tcl,X_\tcl)$. Then, by  Lemma~\ref{lem:lemma3} we obtain
    \begin{align*}
        & \tilde{V}_N(\tcl,X_\tcl) + R_4(\tcl,X_\tcl,M,N)
        = \tilde{J}_N(\tcl,X_\tcl,\hat{\mathbf{U}}) 
        = \tilde{\ell}(\tcl,X_\tcl,\mu_N(X_\tcl)) +  \tilde{J}_{N-1}(\tcl+1,X^+_\tcl,\hat{\mathbf{U}}(\cdot+1)).
    \end{align*}
    Moreover, since $\hat{\mathbf{U}}(\cdot+1)$ corresponds to the optimal control sequence $\mathbf{U}^*_{N-1,X^+}$ up to horizon $M-1$, according to the stochastic dynamic programming principle we can apply Lemma~\ref{lem:lemma3} again to $\tilde{J}_{N-1}(\tcl+1,X^+,\hat{\mathbf{U}}(\cdot+1))$ and get
    \begin{equation} \label{eq:firstEqToDecrease}
        \begin{split}
            \tilde{V}_N(\tcl&,X_\tcl) + R_4(\tcl,X_\tcl,M,N) \\
        &= 
        \tilde{\ell}(\tcl,X_\tcl,\mu_N(X_\tcl)) +  \tilde{V}_{N-1}(\tcl+1,X^+_\tcl) + R_4(\tcl+1,X^+_\tcl,M-1,N-1).
        \end{split}
    \end{equation}
    Furthermore, applying Lemma~\ref{lem:lem2} to equation~\eqref{eq:firstEqToDecrease} in order to estimate the optimal value function with the optimization horizon $N-1$ by the optimal value function with the optimization horizon $N$ leads to 
    \begin{align}
        \tilde{V}_{N}(\tcl+1,X^+_\tcl) =& \tilde{V}_N(\tcl,X_\tcl) - \tilde{\ell}(\tcl,X_\tcl, \mu_N(X_\tcl)) - R_2(\tcl+1,X^+_\tcl,M-1,N-1) \nonumber \\
        &- R_4(\tcl+1,X^+_\tcl,M-1,N-1) + R_4(\tcl,X_\tcl,M,N). \label{eq:theorem2INeq}
    \end{align}
    Now, consider $\theta_2 := \alpha^{-1}(r)$ with $r$ from Assumption~\ref{ass:BoundedStorage_CheapReach}. 
    Then, analogous to equation~\eqref{eq:theta1} for all $X_\tcl \in \tilde{V}^{-1}_N(\tcl,[0,\theta_2])$ we obtain $X_\tcl \in \Bcl^d_r(X^s(\tcl))$ and
    \begin{equation*}
        \theta_2 \geq \tilde{V}_N(X_\tcl) \geq \tilde{\ell}(\tcl+1,X^+_\tcl,U^*_N(1)) \geq \alpha(d(X^+_\tcl,X^s(\tcl+1)))
    \end{equation*}
    which implies $X^+_\tcl \in \Bcl^d_r(X^s(\tcl+1))$.
    Thus, choosing  $L \geq \lceil \vartheta^{-1}(\min\{\eps,\tilde{\eps}\}) \rceil$ and using the estimates on the residual terms from Lemma~\ref{lem:lem2} and Lemma~\ref{lem:lemma3} we obtain that
    \begin{align}
        & -R_2(\tcl+1,X^+_\tcl,M-1,N-1) - R_4(\tcl+1,X^+_\tcl,M-1,N-1) + R_4(\tcl,X_\tcl,M,N) \nonumber \\
        & \leq |R_2(\tcl+1,X^+_\tcl,M-1,N-1)| + |R_4(\tcl+1,X^+_\tcl,M-1,N-1)| + |R_4(\tcl,X_\tcl,M,N)| \nonumber \\
        & \leq 8\gamma_{\tilde{V}}(\vartheta(L))+ 4\gamma_{\lambda}(\vartheta(L)) + 4\gamma_V(N-M,\vartheta(L)) \label{eq:estimateResiduum}
    \end{align}
    holds for all 
    $M \in \{0,\ldots,N\} \setminus ( \bar{\mathcal{Q}}(\tcl+1,X^+_\tcl,L,N-1) \cup \hat{\mathcal{Q}}(\tcl+1,X^+_\tcl,L,N-1) \cup \hat{\mathcal{Q}}(\tcl,X_\tcl,L,N) )$.
    Because each of the sets $\bar{\mathcal{Q}}(\tcl+1,X^+_\tcl,L,N-1)$, $\hat{\mathcal{Q}}(\tcl+1,X^+_\tcl,L,N-1)$, $\hat{\mathcal{Q}}(\tcl,X_\tcl,L,N)$ contains at most $2L$ elements, we can guarantee for $N \geq 12L$ that there exists at least one such $M$ satisfying $M \leq \frac{N}{2}$ and $N-M \geq \frac{N}{2}$.
    Hence, for all $N \geq N_0 := 12 \lceil \vartheta^{-1}(\min\{\eps,\tilde{\eps}\}) \rceil$ an upper bound $\nu(N)$ of \eqref{eq:estimateResiduum} only depending on $N$ is given by 
    \begin{equation}
        \nu(N) := 8\gamma_{\tilde{V}}\left( \vartheta \left( \dfrac{N}{12} \right) \right)+ 4\gamma_{\lambda} \left( \vartheta \left( \dfrac{N}{12} \right) \right) + 4\gamma_V \left( \frac{N}{2}, \vartheta \left( \dfrac{N}{12} \right) \right). \label{eq:upperBoundNu}
    \end{equation}
    Using this estimate, the upper bound on $\tilde{V}_N$, and dissipativity in equation~\eqref{eq:theorem2INeq} we obtain the inequality 
    \begin{equation*}
        \tilde{V}_N(\tcl+1,X^+_\tcl) \leq \tilde{V}_N(\tcl,X_\tcl) - \eta(\tilde{V}_N(\tcl,X_\tcl)) + \nu(N)
    \end{equation*}
    with $\eta := \alpha \circ \alpha_2^{-1} \in \K_{\infty}$. 
    Now define $\mathcal{P}(\tcl) := \tilde{V}_N^{-1}(\tcl,[0,\delta(N)])$ with 
    $$\delta(N) := \max \{ \eta^{-1}(2\nu(N)), \eta^{-1}(\nu(N)) + \nu(N) \}.$$
    Then, for all $X_\tcl \in \tilde{V}_N^{-1}(\tcl,[0,\theta_2]) \backslash \mathcal{P}(\tcl)$ the inequality
    \begin{equation*}
        \tilde{V}_{N}(\tcl,X_\tcl) \geq \delta(N) \geq \eta^{-1}(2\nu(N))
    \end{equation*}
    holds, which implies
    \begin{equation}
        \nu(N) \leq \frac{\eta(\tilde{V}_{N}(\tcl,X_\tcl))}{2}.
    \end{equation}
    Thus, using the lower bound on $\tilde{V}_N$ we obtain
    \begin{align}
        \tilde{V}_{N}(\tcl+1,X^+_\tcl) & \leq \tilde{V}_{N}(\tcl,X_\tcl) - \eta(\tilde{V}_{N}(\tcl,X_\tcl)) + \nu(N) \label{eq:eq1}\\
        & \leq \tilde{V}_{N}(\tcl,X_\tcl) - \dfrac{\eta(\tilde{V}_{N}(\tcl,X_\tcl))}{2} \nonumber \\
        & \leq \tilde{V}_{N}(\tcl,X_\tcl) - \dfrac{\eta(\alpha_1(d(X_\tcl,X^s(\tcl))))}{2}, \label{eq:eq2}
    \end{align}
    which shows the decrease property of the optimal value function for all $X_\tcl \in \tilde{V}_N^{-1}(\tcl,[0,\theta_2]) \backslash \mathcal{P}(\tcl)$ with $\alpha_V(r) = \nu(\alpha_1(r))/2$.
    Thus, for $\mathcal{P}(\tcl) = \tilde{V}_N^{-1}(\tcl,[0,\delta(N)])$ and $\mathcal{Y}(\tcl) = \tilde{V}_N^{-1}(\tcl,[0,\theta])$ with $\theta \leq \theta_0 := \min\{\alpha^{-1}(r),\alpha^{-1}(\tilde{r})\}$ all conditions from Definition~\ref{defn:lyapunovFunction} are satisfied.
    It remains to show that the sets $\mathcal{Y}(\tcl)$ and $\mathcal{P}(\tcl)$ are forward invariant.
    We start with the set $\mathcal{P}(\tcl)$. Let $X_\tcl \in \mathcal{P}(\tcl)$, and hence $\tilde{V}_N(\tcl,X_\tcl) \leq \delta(N)$. In order to show that $X^+_\tcl \in \mathcal{P}(\tcl+1)$, we distinguish two different cases.
    \begin{enumerate}[1.)]
        \item Assume that $\eta(\tilde{V}_N(\tcl,X_\tcl)) \geq \nu(N)$ holds. Then, by \eqref{eq:eq1} we obtain
        $$ \tilde{V}_{N}(\tcl+1,X^+_\tcl) \leq \tilde{V}_{N}(\tcl,X_\tcl) - \eta(\tilde{V}_{N}(\tcl,X_\tcl)) + \nu(N) \leq \tilde{V}_{N}(\tcl,X_\tcl) \leq \delta(N).$$
        \item Assume that $\eta(\tilde{V}_N(\tcl,X_\tcl)) < \nu(N)$ holds. Then, using \eqref{eq:eq1} we get
        \begin{align*}
            \tilde{V}_N(\tcl+1,X^+_\tcl) & \leq \tilde{V}_N(\tcl,X_\tcl) - \eta(\tilde{V}_N(\tcl,X_\tcl)) + \nu(N) \\
            & \leq \tilde{V}_N(\tcl,X_\tcl) + \nu(N) \\
            & \leq \eta^{-1}(\nu(N)) + \nu(N)
            \leq \delta(N).
    \end{align*}
    \end{enumerate}
    Hence, in both cases we receive $\tilde{V}_N(\tcl+1,X^+_\tcl) \leq \delta(N)$ which implies $X^+_\tcl \in \mathcal{P}(\tcl+1)$, and thus, shows the forward invariance of $\mathcal{P}(\tcl)$.
    Now consider $X_\tcl \in \mathcal{Y}(\tcl) \setminus \mathcal{P}(\tcl)$. Then, by equation~\eqref{eq:eq2} we obtain
    \begin{equation*}
        \tilde{V}_{N}(\tcl+1,X^+) \leq \tilde{V}_{N}(\tcl,X_\tcl) - \dfrac{\eta(\alpha_1(d(X,X^s(\tcl))))}{2} \leq \tilde{V}_{N}(\tcl,X_\tcl) \leq \theta,
    \end{equation*}
    which implies $X^+_\tcl \in \mathcal{Y}(\tcl+1)$, and thus, together with the invariance of $\mathcal{P}(\tcl)$ proves the  forward invariance of $\mathcal{Y}(\tcl)$.
\end{proof}

Note that Theorem~\ref{theorem2} shows local practical asymptotic stability, since our assumptions hold only locally.
However, if we modify Assumption~\ref{ass:Dissipativity} and Assumption~\ref{ass:ContinuityModified}(ii) in an appropriate way, we can extend this result in the following semi-global way.

\begin{coro}\label{cor:semiprac}
    Let the Assumptions~\ref{ass:Dissipativity}, \ref{ass:BoundedStorage_CheapReach}, \ref{ass:ContinuityV}, \ref{ass:modifiedStagecost}, and \ref{ass:ContinuityModified} hold. 
    Moreover, consider that the Assumptions~\ref{ass:BoundedStorage_CheapReach} and \ref{ass:ContinuityModified} hold for all $r,\tilde{r} > 0$.\\
    Then, the stationary process $\mathbf{X}^s$ is semi-globally practically asymptotically stable with respect to the optimization horizon $N$, i.e.,
    there exists $\beta \in \K\mathcal{L}$ such that the following property holds: 
    For each $\Delta_1 > 0$ and $\Delta_2 > \Delta_1$ there exists $N_{\Delta} \in \N$ such that for all $N \geq N_{\Delta}$ and all $X_{\tcl_0} \in \Bcl_{\Delta_2}^d(X^s(\tcl_0))$ the inequality
    \begin{equation}
        d(X_{\mu_N}(\tcl; \tcl_0, X_{\tcl_0}), X^s(\tcl+\tcl_0)) \leq \max\{\beta(d(X_{\tcl_0},X^s(\tcl_0)), \tcl), \Delta_1\}
    \end{equation}
    holds for all $\tcl_0,\tcl \in \mathbb{N}_0$.
\end{coro}
\begin{proof}
    Since the Assumptions~\ref{ass:BoundedStorage_CheapReach} and \ref{ass:ContinuityModified} hold for all $r,\tilde{r} > 0$, we can set $r = \tilde{r} = \alpha^{-1}(\gamma_{\tilde{V}}(\Delta_2))$ implying that we can choose $\theta = \theta_0 = \gamma_{\tilde{V}}(\Delta_2)$ in Theorem~\ref{theorem2}.
    Further, consider $N_{\Delta_{1,2}} \geq N_0$ large enough such that $\alpha(\delta(N_{\Delta_{1,2}})) \leq \Delta_1$ holds with $N_0 > 0$ and $\delta \in \mathcal{L}$ from Theorem~\ref{theorem2}.
    Then, the statement follows directly by Theorem~\ref{theorem2}, since these choices imply that $\Bcl^d_{\Delta_2}(X^s(\tcl)) \subseteq \mathcal{Y}(\tcl)$ and $\mathcal{P}(\tcl) \subseteq \Bcl^d_{\Delta_1}(X^s(\tcl))$ hold for all $\tcl \in \N_0$ and $\mathcal{Y}(\tcl), \mathcal{P}(\tcl)$ from Theorem~\ref{theorem2}.
\end{proof}

The results of this section hold for arbitrary (pseudo)metrics, provided all assumptions are satisfied with respect to the same metric. 
However, as noted when introducing the notions of stochastic turnpike and dissipativity, the specific choice of the metric affects the strength of the resulting stability statements.
For example, if $d$ is the $L^p$-norm from equation~\eqref{eq:LpNorm}, we obtain $p$-th mean stability, which aligns with convergence in the $p$-th mean.
Similarly, the Ky-Fan metric from equation~\eqref{eq:KyFanMetric} metricizes convergence in probability, and the Lévy-Prokhorov metric from equation~\eqref{eq:LevyProhkorovMetric} metricizes weak convergence of measures, i.e.\ convergence in distribution.
Hence, the stability properties we derive are consistent with classical stochastic convergence concepts.
\section{Near-optimal performance of stochastic MPC} \label{sec:Performance}
After investigating the stability properties of the closed-loop solution, we now aim to show that the solutions approximated by our proposed MPC algorithm also have near-optimal performance.
However, before starting, we have to clarify how we measure optimality on the infinite time horizon.
The problem that arises here is that, in general, the optimal cost on an infinite horizon will not be finite, even if $\vert V_N(j,X) \vert < \infty$ holds for all $N \in \mathbb{N}$.
Thus, one has to find a way to determine whether a control sequence has better performance on an infinite horizon than another, although both incur infinite cost.
In our case, we use the following definition of overtaking optimality, originally introduced in \cite{gale1967}.

\begin{defi}[Overtaking optimality] \label{defn:overtakingOptimality}
    Let $X_\tcl \in \RR{\Omega,\X}$ and consider a control sequence $\mathbf{U}^*_{\infty} \in \UU^{\infty}_{ad}(\tcl,X_\tcl)$ with corresponding state trajectory $X_{\mathbf{U}^*_{\infty}}(\cdot;\tcl,X_\tcl)$. Then, the pair $(X_{\mathbf{U}^*_{\infty}}(\cdot;\tcl,X_\tcl),\mathbf{U}^*_{\infty})$ is called overtaking optimal if
    \begin{equation}
        \liminf_{K \to \infty} \left( \sum_{\tol=0}^{K-1} \ell(X_{\mathbf{U}}(\tol;\tcl,X_\tcl),U(\tol)) - \ell(X_{\mathbf{U}^*_{\infty}}(\tol;\tcl,X_\tcl),U^*_{\infty}(\tol)) \right) \geq 0
    \end{equation}
    holds for all $\mathbf{U} \in \UU^{\infty}_{ad}(\tcl,X_\tcl)$.
\end{defi}

The above definition can be interpreted in the following way: 
A control sequence is overtaking optimal if for all $\eps > 0$ there exists a $K_{\eps} \in \N$ such that for all $K \geq K_{\eps}$, the performance of this sequence is at most $\eps$ worse than that of any other admissible control sequence.

However, while this definition enables us to define an optimal control sequence even in the case of infinite costs, we cannot use it to formulate a cheap reachability condition, which is necessary to derive the infinite-horizon turnpike property from our dissipativity assumption.
To resolve this problem, we introduce the following shifted optimal control problem.

\begin{defi}[Shifted optimal control problem]
    Assume that the optimal control problem \eqref{eq:stochOCPopenloop} is stochastically dissipative at $(\mathbf{X}^s,\mathbf{U}^s)$.
    Then we define the shifted optimal control problem as 
    \begin{equation} \label{eq:stochOCPshifted}
        \begin{split}
            \minimize_{\mathbf{U} \in \UU_{ad}^{N}(\tcl,X_\tcl)} &~\hat{J}_N(X_\tcl,\mathbf{U}) := \sum_{\tol=0}^{N-1} \hat{\ell}(X(\tol;\tcl,X_\tcl),U(\tol))
        \end{split}
    \end{equation}
    with $\hat{\ell}(X,U) := \ell(X,U) - \ell(\mathbf{X}^s,\mathbf{U}^s)$ and corresponding optimal value function
    $\hat{V}_N(\tcl,X) := \inf_{\mathbf{U} \in \UU_{ad}^{N}(\tcl,X)} \hat{J}_N(X,\mathbf{U})$.
\end{defi}

Note that, for finite horizons, the optimal solution to the shifted problem \eqref{eq:stochOCPshifted} and the original problem~\eqref{eq:stochOCPopenloop} will coincide, since $\ell(\mathbf{X}^s,\mathbf{U}^s)$ is only a constant shift.
Moreover, on the infinite horizon, the minimizing solution to \eqref{eq:stochOCPshifted} will be overtaking optimal in the sense of Definition~\ref{defn:overtakingOptimality}.

Similar to Assumption~\ref{ass:ContinuityV} and Assumption~\ref{ass:ContinuityModified}(ii), we impose a continuity property on the shifted cost on the infinite horizon.
Note that the continuity of the finite-horizon optimal cost is guaranteed by Assumption~\ref{ass:ContinuityV}, since $$\hat{V}_N(\tcl,X) - \hat{V}_N(\tcl,X^s(\tcl)) = V_N(\tcl,X) - V_N(\tcl,X^s(\tcl))$$ holds.

\begin{ass} \label{ass:ContinuityVshifted}
    Consider the stationary pair $(\mathbf{X}^s,\mathbf{U}^s)$ and the (pseudo)metric $d$ from Assumption~\ref{ass:Dissipativity}. 
    Then, we assume that $V_{\infty}$ is continuous at $\mathbf{X}^s$, i.e., there exist $\hat{\eps} > 0$, $\gamma_{V_{\infty}} \in \K$ such that
    $$\vert \hat{V}_{\infty}(\tcl,X_\tcl) - \hat{V}_{\infty}(\tcl,X^s(\tcl)) \vert \leq \gamma_{V_{\infty}}(d(X_\tcl,X^s(\tcl)))$$
    holds for all $\tcl \in \N_0$ and $X_\tcl \in \Bcl_{\hat{\eps}}^d(X^s(\tcl))$.
\end{ass}

Using the above assumption we can now show that the shifted optimal cost is indeed finite in our framework.

\begin{lem} \label{lem:finiteV}
    Let Assumption~\ref{ass:Dissipativity}, Assumption~\ref{ass:BoundedStorage_CheapReach}, and Assumption~\ref{ass:ContinuityVshifted} hold. 
    Then, there exists $C_{V_{\infty}} > 0$ such that for all $(\tcl,X_\tcl) \in \Bcl^d_r(\mathbf{X}^s)$ it holds that
    $\vert \hat{V}_{\infty}(\tcl,X_\tcl) \vert \leq C_{V_{\infty}}$.
\end{lem}
\begin{proof}
    Consider $(\tcl,X_\tcl) \in \Bcl^d_r(\mathbf{X}^s)$, $\hat{\eps}$ from the continuity property on infinite horizon from Assumption~\ref{ass:ContinuityVshifted}, $r$ from Assumption~\ref{ass:BoundedStorage_CheapReach}(i) and set $\delta := \min\{\hat{\eps},r\}$.
    Now we pick $L \in \N$ such that $\vartheta(L) < \delta$ with $\vartheta \in \mathcal{L}$ from Definition~\ref{thm:Turnpike}. 
    Thus, we can conclude by the turnpike property, which holds on locally on $\Bcl_r^d(X^s(\tcl))$ due to Lemma~\ref{thm:Turnpike} that for all $N > L$ there is $M \leq N$ such that $X_{\mathbf{U}^*_{N}}(M;\tcl,X_\tcl) \in \Bcl_{\delta}^d(X^s(M+\tcl))$.
    Hence, we obtain by the continuity property that 
    \begin{equation*}
    \begin{split}
        \vert \hat{V}_{\infty}(\tcl,&X_{\mathbf{U}^*_{N}}(M;\tcl,X_\tcl)) - \hat{V}_{\infty}(\tcl,X^s(\tcl)) \vert
        \leq \gamma_{V_\infty}\left(d ( X_{\mathbf{U}^*_{N}}(M;\tcl,X_\tcl), X^s(\tcl+M) )\right) < \gamma_{V_\infty}(\delta).
    \end{split}
    \end{equation*}
    By optimality this implies
    \begin{equation*}
    \begin{split}
        \vert \hat{V}_{\infty}(\tcl,X_\tcl) \vert \leq& \vert \hat{J}_M(X_\tcl,\mathbf{U}^*_{N}) + \hat{V}_{\infty}(\tcl,X_{\mathbf{U}^*_{N}}(M;\tcl,X_\tcl)) - \hat{V}_{\infty}(\tcl,X^s(\tcl)) + \hat{V}_{\infty}(\tcl,X^s(\tcl)) \vert \\
        <& \vert \hat{J}_M(X_\tcl,\mathbf{U}^*_{N}) \vert + \vert \hat{V}_{\infty}(X_{\mathbf{U}^*_{N}}(M;\tcl,X_\tcl)) - \hat{V}_{\infty}(\tcl,X^s(\tcl)) \vert + \vert \hat{V}_{\infty}(\tcl,X^s(\tcl)) \vert \\
        <& \vert \hat{J}_M(X_\tcl,\mathbf{U}^*_{N}) \vert + \vert \hat{V}_{\infty}(\tcl,X^s(\tcl)) \vert + \gamma_{V_\infty}(\delta).
    \end{split}
    \end{equation*}
    Furthermore, we know that $\hat{V}_{\infty}(\tcl,X^s(\tcl)) \leq 0$ since $\hat{\ell}(X^s(\tcl+\tol),U^s(\tcl+\tol)) = 0$ for all $\tol \in \N_0$ and by dissipativity for all $N \in \N$ we obtain
    \begin{equation*}
        \hat{V}_{N}(\tcl,X^s(\tcl)) \geq -\lambda(j,X^s(j)) - C^l_{\lambda} \geq - (C_{\lambda} + C^l_{\lambda})
    \end{equation*}
    with $C_{\lambda}$ from Assumption~\ref{ass:BoundedStorage_CheapReach}(i) which implies that $\vert \hat{V}_{\infty}(\tcl,X^s(\tcl)) \vert \leq C_{\lambda} + C^l_{\lambda}$.
    Moreover, by using the finite horizon DPP from Theorem~\ref{thm:FiniteDPP} and the cheap reachability on $\Bcl^d_r(\mathbf{X}^s)$ from Assumption~\ref{ass:BoundedStorage_CheapReach}(ii) we get
    \begin{equation*}
    \begin{split}
        \vert \hat{J}_M(X_\tcl,\mathbf{U}^*_{N}) \vert 
        &= \vert \hat{V}_N(\tcl,X_\tcl) - \hat{V}_{N-M}(\tcl+M,X_{\mathbf{U}^*_{N}}(M;\tcl,X_\tcl)) \vert \\
        &\leq \vert \hat{V}_N(\tcl,X_\tcl) \vert + \vert \hat{V}_{N-M}(\tcl+M,X_{\mathbf{U}^*_{N}}(M;\tcl,X_\tcl)) \vert \leq 2 C_V.
    \end{split}
    \end{equation*}
    Thus, the claim follows with $C_{V_{\infty}} := 2 C_V + C_{\lambda} + C^l_{\lambda} + \gamma_{V_{\infty}}(\delta) > 0$.
\end{proof}

Note that the finiteness of the shifted optimal value function from Lemma~\ref{lem:finiteV} in particular implies that the inequality
\begin{equation*}
    \hat{V}_{\infty}(\tcl,X_\tcl) = \sum_{\tol=0}^{\infty} \left( \ell(X_{\mathbf{U}_{\infty}^*}(\tol;\tcl,X_\tcl), U(\tol)) - \ell(\mathbf{X}^s,\mathbf{U}^s) \right) \leq C_{V_{\infty}}
\end{equation*}
holds for all $(\tcl,X_\tcl) \in \Bcl^d_r(\mathbf{X}^s)$, which can be interpreted as a cheap reachability condition on infinite horizon. 
Hence, we can show that the optimal control problem~\eqref{eq:stochOCPopenloop} also has the infinite-horizon turnpike property which concludes the preliminaries of this section.

\begin{lem} \label{lem:TurnpikeInfty}
    Let the Assumptions~\ref{ass:Dissipativity}, \ref{ass:BoundedStorage_CheapReach}, and \ref{ass:ContinuityVshifted} hold.
    Then, the stochastic optimal control problem~\eqref{eq:stochOCPopenloop} has the uniform stochastic finite-horizon turnpike property on $\Bcl^d_r(\mathbf{X}^s)$ with respect to the (pseudo)metric $d$.
\end{lem}
\begin{proof}
    The proof follows analogously to the finite-horizon case from Lemma~\ref{thm:Turnpike}.
\end{proof}

\subsection{Non-averaged performance}

    First, we aim to obtain estimates for the non-averaged performance of our proposed MPC algorithm.
    To this end, we start with the following results, which shows that the closed-loop trajectory is also finite-horizon near-optimal compared to all other trajectories that have the same stability properties.

    \begin{thm} \label{thm:Performance5}
        Let the Assumptions~\ref{ass:Dissipativity}, \ref{ass:BoundedStorage_CheapReach}, \ref{ass:ContinuityV}, \ref{ass:modifiedStagecost}, \ref{ass:ContinuityModified}, and \ref{ass:ContinuityVshifted} hold.
        Consider $\alpha \in \K_{\infty}$ from Assumption~\ref{ass:Dissipativity}, $\delta \in \mathcal{L}$, $N_0 \in \N$ and $\theta_0 >0$ from Theorem~\ref{theorem2}, $\beta \in \K\mathcal{L}$ from \eqref{eq:stability}, $\nu \in \mathcal{L}$ from equation~\eqref{eq:upperBoundNu}, and for $\kappa \geq 0$ define
        \begin{equation}
            \UU_{ad}^{K,\kappa}(X_0) := \{ \mathbf{U} \in \UU_{ad}^K(0,X_0) \mid X_{\mathbf{U}}(K;0,X_0) \in \Bcl^d_{\kappa}(X^s(K)) \}.
        \end{equation}
        Then, there exist $\delta_2,\delta_3 \in \mathcal{L}$ such that for all sufficiently large $N,K \in \N$ and all $X_0 \in \tilde{V}_N^{-1}([0,\theta_0])$ the inequality 
        \begin{equation*} 
            J_K^{cl}(X_0,\mu_N) \leq \inf_{\mathbf{U} \in \UU_{ad}^{K,\kappa}(X_0)} J_K^{cl}(X_0,\mathbf{U}) + \delta_2(N) + K\nu(N) + \delta_3(K).
        \end{equation*}
        with $\kappa = \max \{ \beta(d(X_0,X^s(0)),K), \alpha^{-1}(\delta(N)) \}$.
    \end{thm}
    \begin{proof}
        Consider $X_\tcl \in \tilde{V}_N^{-1}(\tcl,[0,\theta_0])$ and set $X^+_\tcl = f(X\tcl,\mu_N(X_\tcl),W(\tcl))$.
        Then, from equation~\eqref{eq:theorem2INeq} we obtain for all $N \geq N_0$ the identity 
        \begin{align*}
           \tilde{\ell}(\tcl, X_\tcl, \mu_N(X_\tcl)) =& \tilde{V}_N(\tcl,X_\tcl) - \tilde{V}_{N}(\tcl+1,X^+_\tcl) - R_2(\tcl+1,X^+_\tcl,M-1,N-1) \nonumber \\
            & - R_4(\tcl+1,X^+_\tcl,M-1,N-1) + R_4(\tcl,X_\tcl,M,N).
        \end{align*}
        with $\vert R_4(\tcl,X_\tcl,M,N)- R_4(\tcl+1,X^+_\tcl,M-1,N-1)- R_2(\tcl+1,X^+_\tcl,M-1,N-1) \vert \leq \nu(N)$
        and $\nu$ from equation~\eqref{eq:upperBoundNu}.
        Hence, using that $\tilde{V}_N^{-1}(\tcl,[0,\theta_0])$ is a forward invariant set for the closed -loop system and $X_0 \in \tilde{V}_N^{-1}(0,[0,\theta_0])$ holds, summing the cost along the closed-loop trajectory yields
        \begin{equation} \label{eq:upperBoundModifiedCost}
            \sum_{\tcl=0}^{K-1} \tilde{\ell}(X_{\mu_N}(\tcl),\mu_N(X_{\mu_N}(\tcl))) \leq \tilde{V}_N(0,X_0) - \tilde{V}_N(K,X_{\mu_N}(K)) + K \nu(N).
        \end{equation}
        Now consider $N,K \in \N$ such that $\max \{ \beta(d(X_0,X^s(0)),K), \alpha^{-1}(\delta(N)) \} \leq \tilde{\eps}$ holds with $\tilde{\eps}$ from Assumption~\ref{ass:ContinuityModified}(ii). 
        Thus, if $K \leq N$ holds, we can conclude by Assumption~\ref{ass:ContinuityModified} that
        \begin{equation*} 
        \begin{split}
            \tilde{J}_K(X_0,\mathbf{U}) =& \tilde{J}_K(X_0,\mathbf{U}) + \tilde{V}_{N-K}(K,X_{\mathbf{U}}(K)) - \tilde{V}_{N-K}(K,X_{\mathbf{U}}(K)) 
            \geq \tilde{V}_N(X_0) - \gamma_{\tilde{V}}(\kappa)
        \end{split}
        \end{equation*}
        holds for all $\mathbf{U} \in \UU_{ad}^{K,\kappa}(X_0)$.
        On the other hand, if $K \geq N$ it holds by the non-negativity of $\tilde{\ell}$ that $\tilde{J}_K(X_0,\mathbf{U}) \geq \tilde{V}_N(X_0)$.
        Hence, using Assumption~\ref{ass:ContinuityModified}(i) we obtain
        \begin{equation*} 
        \begin{split}
            J_K^{cl}(X_0,\mu_N) =& \sum_{\tcl=0}^{K-1} \tilde{\ell}(\tcl, X_{\mu_N}(\tcl), \mu_N(X_{\mu_N}(\tcl))) - \lambda(0,X_0) + \lambda(K,X_{\mu_N}(K)) + \ell(\mathbf{X}^s,\mathbf{U}^s)\\
            \leq& \tilde{V}_N(0,X_0) - \tilde{V}_N(K,X_{\mu_N}(K)) + K \nu(N) \\
            &- \lambda(0,X_0) + \lambda(K,X_{\mu_N}(K)) + \ell(\mathbf{X}^s,\mathbf{U}^s)\\
            \leq& \tilde{J}_K(X_0,\mathbf{U}) + \gamma_{\tilde{V}}(\kappa) - \tilde{V}_N(K,X_{\mu_N}(K)) + K \nu(N) \\
            &- \lambda(0,X_0) + \lambda(K,X_{\mu_N}(K)) + \ell(\mathbf{X}^s,\mathbf{U}^s) \\
            \leq& J_K(X_0,\mathbf{U}) + \gamma_{\tilde{V}}(\kappa) - \tilde{V}_N(K,X_{\mu_N}(K)) + K \nu(N) \\
            &- \lambda(X_{\mathbf{U}}(K)) + \lambda(K,X_{\mu_N}(K)) \\
            \leq& J_K(X_0,\mathbf{U}) + \gamma_{\tilde{V}}(\kappa) + K \nu(N) + 2 \gamma_{\lambda}(\kappa).
        \end{split}
        \end{equation*}
        Using the definition of $\kappa$ and $X_{\mu_N}(\tcl) \in \tilde{V}_N^{-1}([0,\theta_0]) \subseteq \Bcl_r^d(X^s(\tcl))$ we can estimate
        \begin{equation*}
            \gamma_{\tilde{V}}(\kappa) + 2 \gamma_{\lambda}(\kappa) \leq \delta_2(N) + \delta_3(K)
        \end{equation*}
        with $\delta_2(N) := \gamma_{\tilde{V}}(\alpha^{-1}(\delta(N))) + 2 \gamma_{\lambda}(\alpha^{-1}(\delta(N)))$ and $\delta_3(K) := \gamma_{\tilde{V}}(\beta(r,K)) + 2 \gamma_{\lambda}(\beta(r,K))$, which proves the theorem.
    \end{proof}

    Using Theorem~\ref{thm:Performance5} and the infinite-horizon turnpike property from Lemma~\ref{lem:TurnpikeInfty} we can additionally show that the closed-loop solution is approximately overtaking optimal, which characterizes the performance of our stochastic MPC scheme on the infinite-horizon.
    
    \begin{coro}
        Let the Assumptions~\ref{ass:Dissipativity}, \ref{ass:BoundedStorage_CheapReach}, \ref{ass:ContinuityV}, \ref{ass:ContinuityModified}, \ref{ass:modifiedStagecost}, and \ref{ass:ContinuityVshifted} hold.
        Consider $\nu \in \mathcal{L}$ from equation~\eqref{eq:upperBoundNu} and $\delta_2 \in \mathcal{L}$ from Theorem~\ref{thm:Performance5}.
        Then, for all sufficiently large $N$ and all $X_0 \in \tilde{V}^{-1}_N(0,[0,\theta_0])$ the inequality 
        \begin{equation*}
        \begin{split}
            \liminf_{K \rightarrow \infty} \bigg( \sum_{\tcl=0}^{K-1} \Big(\ell(&X_{\mathbf{U}}(\tcl;0,X_0),U(\tcl))
            - \ell(X_{\mu_N}(\tcl;0,X_0),\mu_N(X_{\mu_N}(\tcl))) \Big)
            + K \nu(N) + \delta_2(N) \bigg) \geq 0
        \end{split}
        \end{equation*}
        holds for all $\mathbf{U} \in \UU^{\infty}_{ad}(X_0)$.
        \label{cor:overtaking}
    \end{coro}
    \begin{proof}
        Choose, $N,K \in \N$ large enough such that Theorem~\ref{thm:Performance5} holds and consider the infinite-horizon optimal control sequence $\mathbf{U}^*_{\infty}$. 
        Due to the infinite-horizon turnpike property, cf.\ Lemma~\ref{lem:TurnpikeInfty}, we can conclude that there exist a $K_0 \geq K$ such that $d(X_{\mathbf{U}^*_{\infty}}(\tilde{K};0,X_0),X^s(\tilde{K})) \leq \alpha^{-1}(\delta(N))$ and $\beta(d(X_0,X^s(0)),\tilde{K}) \leq \alpha^{-1}(\delta(N))$ holds for all $\tilde{K} \geq K_0$.
        Hence, we obtain that $\mathbf{U}^*_{\infty} \in \UU^{\tilde{K},\kappa}_{ad}$ holds and we can apply Theorem~\ref{thm:Performance5} to obtain
        \begin{equation*} 
            J_K^{cl}(X_0,\mu_N) \leq J_K^{cl}(X_0,\mathbf{U}^*_{\infty}) + \delta_2(N) + K\nu(N) + \delta_3(K).
        \end{equation*}
        Since $\delta_3(K) \to 0$ for $K \to \infty$, taking the limit yields
        \begin{equation*}
        \begin{split}
            \liminf_{K \rightarrow \infty} \bigg( \sum_{\tcl=0}^{K-1} \Big(\ell(&X_{\mathbf{U}^*_{\infty}}(\tcl),U^*_{\infty}(\tcl)) 
            - \ell(X_{\mu_N}(\tcl),\mu_N(X_{\mu_N}(\tcl))) \Big) + K \nu(N) + \delta_2(N) \bigg) \geq 0.
        \end{split}
        \end{equation*}
        Because of the optimality of $\mathbf{U}^*_{\infty}$ we can further conclude that 
        \begin{equation*}
        \begin{split}
            \liminf_{K \rightarrow \infty} \bigg( \sum_{\tcl=0}^{K-1} \Big(\ell(X_{\mathbf{U}}(\tcl),U(\tcl))- \ell(X_{\mathbf{U}^*_{\infty}}(\tcl),U^*_{\infty}(\tcl)) \Big) \bigg) \geq 0
        \end{split}
        \end{equation*}
        holds for all $\mathbf{U} \in \UU^{\infty}_{ad}(X_0)$, which proves the claim.
    \end{proof}

    \begin{rem}
        One might be worried by the $K$-dependence of the error term $K\nu(N)$ in Corollary~\ref{cor:overtaking}, because this implies that the deviation from the optimal cost grows linearly in $K$. 
        However, as the accumulated non-shifted cost itself grows linearly in $K$ (except in the very particular case when the stationary cost $\ell(\mathbf{X}^s,\mathbf{U}^s)$ equals $0$), the relative deviation to the non-shifted optimal cost is constant in $K$. 
    \end{rem}

\subsection{Averaged performance}

    After obtaining an estimate for the non-averaged performance in Corollary~\ref{cor:overtaking},
    we can use the stability results from Theorem~\ref{thm:stabilityLyapunov} to obtain bounds for the averaged closed-loop performance, 
    defined by 
    \begin{equation*}
        \bar{J}^{cl}_{K}(X,\mu_N) :=  \frac{1}{K} J_K^{cl}(X,\mu_N).
    \end{equation*}
    For this purpose, we first implement a lower bound on the average performance which is valid for all admissible control sequences.

    \begin{lem} \label{lem:optimalOperation}
        Let Assumption~\ref{ass:Dissipativity} hold.
        Then, the inequality
        \begin{equation*}
            \liminf_{K \to \infty} \frac{1}{K} \sum_{\tcl=0}^{K-1} \ell(X_{\mathbf{U}}(\tcl;0,X_0), U(\tcl)) \geq \ell(\mathbf{X}^s,\mathbf{U}^s)
        \end{equation*}
        holds for all $X_0 \in \RR{\Omega,\X}$ and $\mathbf{U} \in \UU_{ad}^{\infty}(X_0)$.
    \end{lem}
    \begin{proof}
    See Appendix~\ref{app}. For a similar proof in the deterministic case see also \cite[Proposition~6.4]{Angeli2012}.
    \end{proof}

    The property from Lemma~\ref{lem:optimalOperation} is also called optimal operation and shows that the asymptotic average performance cannot be better than that of the stationary pair.
    Based on this property, we can now show that the average performance of the closed-loop solution can come arbitrarily close to this lower bound, depending on the optimization horizon $N$.

    \begin{thm} \label{thm:AvgPerformance}
        Let the Assumptions~\ref{ass:Dissipativity}, \ref{ass:BoundedStorage_CheapReach}, \ref{ass:ContinuityV}, \ref{ass:ContinuityModified}, \ref{ass:modifiedStagecost}, and \ref{ass:ContinuityVshifted} hold.
        Then, for all $N \in \N$ sufficiently large and all $X_0 \in \tilde{V}_N^{-1}(0,[0,\theta_0])$ the averaged closed-loop cost satisfies
        \begin{equation*}
            \ell(\mathbf{X}^s,\mathbf{U}^s) \leq 
            \limsup_{K \to \infty} \bar{J}^{cl}_{K}(X_0,\mu_N)
            \leq \ell(\mathbf{X}^s,\mathbf{U}^s) + \nu(N).
        \end{equation*}
        with $\nu(N)$ from equation~\eqref{eq:upperBoundNu} and $\theta_0>0$ from Theorem~\ref{theorem2}.
    \end{thm}
    \begin{proof}
        Using equation~\eqref{eq:upperBoundModifiedCost} we can conclude that
        \begin{equation} \label{eq:eq1averagedPerformance}
            \begin{split}
                \hat{J}_K^{cl}(X_0,\mu_N) =& \sum_{\tcl=0}^{K-1} \tilde{\ell}(\tcl, X_{\mu_N}(\tcl), \mu_N(X_{\mu_N}(\tcl))) - \lambda(0,X_0) + \lambda(K,X_{\mu_N}(K)) \\
                \leq& \tilde{V}_N(0,X_0) - \tilde{V}_N(K,X_{\mu_N}(K)) + K \nu(N) - \lambda(0,X_0) + \lambda(K,X_{\mu_N}(K)) 
            \end{split}
        \end{equation}
        holds. 
        Now consider $N \in \N$ large enough such that $\alpha^{-1}(\delta(N)) \leq \tilde{\eps}$ holds with $\tilde{\eps}$ from Assumption~\ref{ass:ContinuityModified}(ii) and choose $K_0$ such that $\beta(d(X_0,X^s(0)),K_0) \leq \alpha^{-1}(\delta(N))$ holds.
        Then, by the stability property from Theorem~\ref{thm:stabilityLyapunov} for all $K \geq K_0$ we can conclude that $X_{\mu_N}(K) \in \Bcl_{\tilde{\eps}}^d(X^s(\tcl))$.
        Using equation~\eqref{eq:eq1averagedPerformance}, Assumption~\ref{ass:BoundedStorage_CheapReach} and \ref{ass:ContinuityModified}(ii), and the non-negativity of $\tilde{\ell}$ this implies
        \begin{equation*} 
            \begin{split}
                \hat{J}_K^{cl}(X_0,\mu_N)
                \leq& \left( \tilde{V}_N(0,X_0) - \tilde{V}_N(K,X_{\mu_N}(K)) + K \nu(N) - \lambda(0,X_0) + \lambda(K,X_{\mu_N}(K)) \right)\\
                \leq&  \left( \theta_0 + K \nu(N) - \lambda(0,X_0) + \lambda(K,X^s(K)) + \gamma_{\lambda}(\tilde{\eps}) \right) \\
                \leq& \theta_0 + C^l_{\lambda} + C_{\lambda} + \gamma_{\lambda}(\tilde{\eps}) + K\nu(N),
            \end{split}
        \end{equation*}
        where $-C^l_{\lambda}$ is uniform lower bound on the storage function $\lambda$.
        Dividing by $K$ and taking the $\limsup_{K \to \infty}$ then yields 
        \begin{equation}
            \begin{split}
                \limsup_{K \to \infty} \bar{J}^{cl}_{K}(X_0,\mu_N) =& \limsup_{K \to \infty} \frac{1}{K} \hat{J}_K^{cl}(X_0,\mu_N) - \ell(\mathbf{X}^s,\mathbf{U}^s) \\
                \leq& \limsup_{K \to \infty} \frac{\theta_0 + C^l_{\lambda} + C_{\lambda} + \gamma_{\lambda}(\tilde{\eps})}{K} + \nu(N) \\
                =& \nu(N),
            \end{split}
        \end{equation}
        which proves the claim.
    \end{proof}

It should be noted that the non-averaged and averaged performance estimates obtained in this section are directly derived from the stability properties of the closed-loop solution proven in Section~\ref{sec:Stability}.
As a result, the stability characterization from Theorem~\ref{thm:stabilityLyapunov} and the performance estimates from this section use the same function $\nu \in \K_{\infty}$ to define the error terms.
However, it is also possible to derive near-optimal performance bounds without first proving stability.
For such an approach we refer to \cite{Schiessl2024b}, where near-optimality is shown based on turnpike properties and a slightly different and more restrictive optimal operation condition than the one from Lemma~\ref{lem:optimalOperation}.

\section{Numerical simulations} \label{sec:numerics}
To illustrate our theoretical findings, we consider the one-dimensional nonlinear optimal control problem
\begin{equation} \label{eq:example}
\begin{split}
\min_{\mathbf{U} \in \U^{N}(X_0)} J_N(X_0,\mathbf{U}) &:= \sum_{k=0}^{N-1} \Exp{X(k)^2 + \gamma U(k)^2} \\
s.t. ~ X(k+1) &= (U(k)-X(k))^2 + W(k)
\end{split}
\end{equation}
where $\gamma =25$ is a regularization parameter and $W(k)$ follows a two-point distribution such that $W(k) = a := 1$ with probability $p_a = 0.7$ and $W(k) = b := 0.25$ with probability $p_b = 0.3$.
Note that this problem is a modification of the example from \cite[Section~IV]{Schiessl2025} for which strict $L^2$ dissipativity was shown analytically. 
More precisely problem \eqref{eq:example} uses the same dynamics, but slightly different stage costs.
Although for this stage cost we cannot theoretically guarantee stochastic dissipativity anymore, we can still observe stochastic turnpikes numerically, cf.\ Figure~\ref{fig:turnpikes}.
There, we can see that all possible realization paths are close to each other in the middle of the optimization horizon, indicating a pathwise---and thus also a distributional---turnpike property.

\begin{figure}[ht]
    \begin{minipage}{0.49\textwidth}
        \centering
        \includegraphics[width=0.7\textwidth]{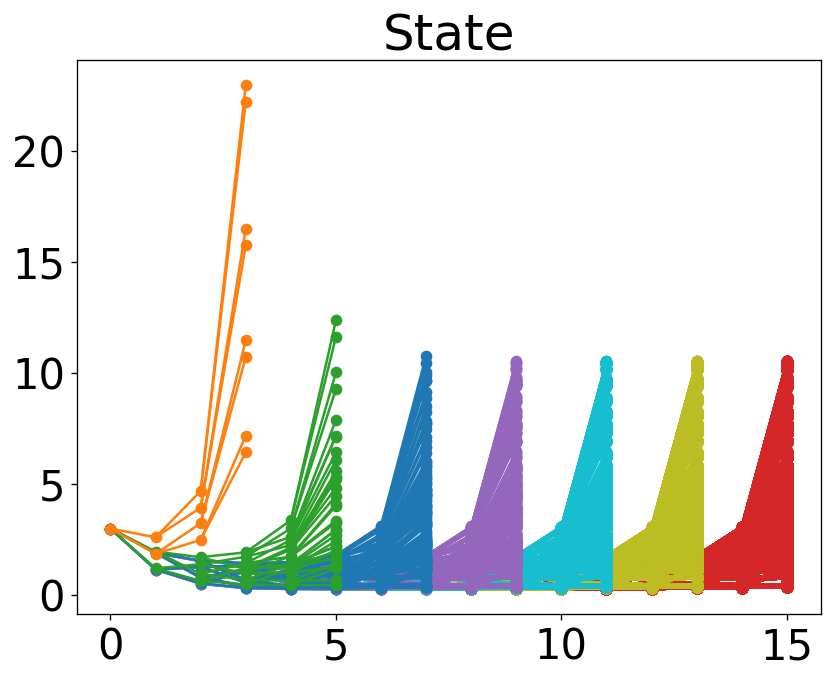}
    \end{minipage}
    \begin{minipage}{0.49\textwidth}
        \centering
        \includegraphics[width=0.7\textwidth]{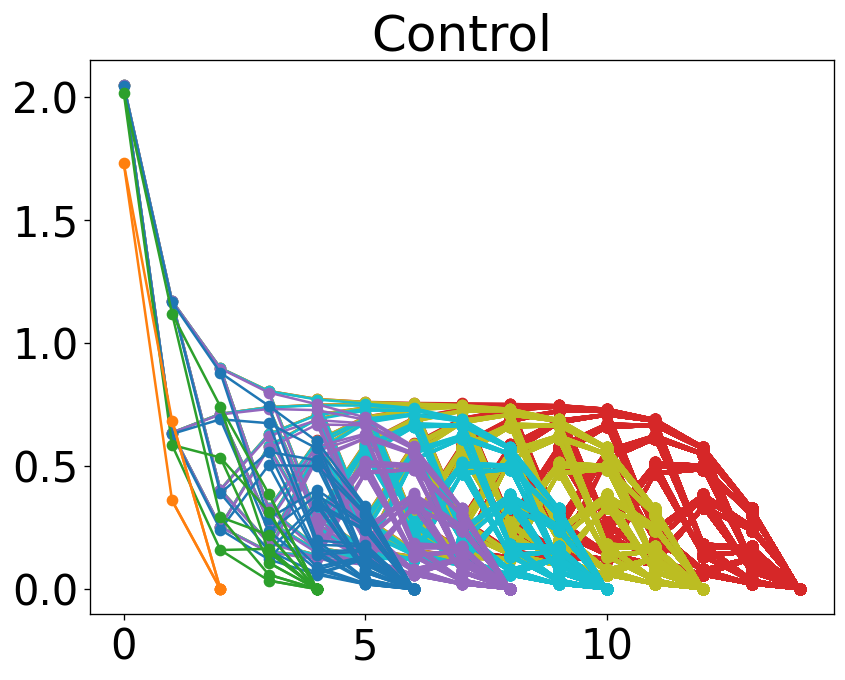}
    \end{minipage}
    \caption{Evolution of all possible realization paths for the optimal trajectories (left) and controls (right) for $N=3,5,7,\ldots,15$.}
    \label{fig:turnpikes}
\end{figure}

Hence, based on our theoretical findings, we expect to observe stability both in a pathwise sense and in distribution.
To this end, we generated $20000$ sample paths of the closed-loop solution up to time $K=150$ using Algorithm~\ref{alg:implementableStochMPC} with the initial condition $X_0 = 3$.
The optimal control problems arising in each MPC iteration are solved by direct minimization over all possible realization paths, which is feasible since the image of the noise process $W(k)$ is finite.
Furthermore, to estimate the characteristics of the stationary stochastic process, we employed the values attained at the middle of the horizon by the solution on horizon $N=15$, shown in red in Figure~\ref{fig:turnpikes}, as suggested by the turnpike property.

Figure~\ref{fig:stability_paths} shows the sample paths of the closed-loop solutions together with the upper and lower bounds of the stationary process realizations. As the horizon increases, the paths converge toward the region of the stationary process realizations, thereby emphasizing the practical asymptotic stability established in Theorem~\ref{theorem2}.

\begin{figure}[ht]
    \begin{minipage}{0.32\textwidth}
        \centering
        \includegraphics[width=0.95\textwidth]{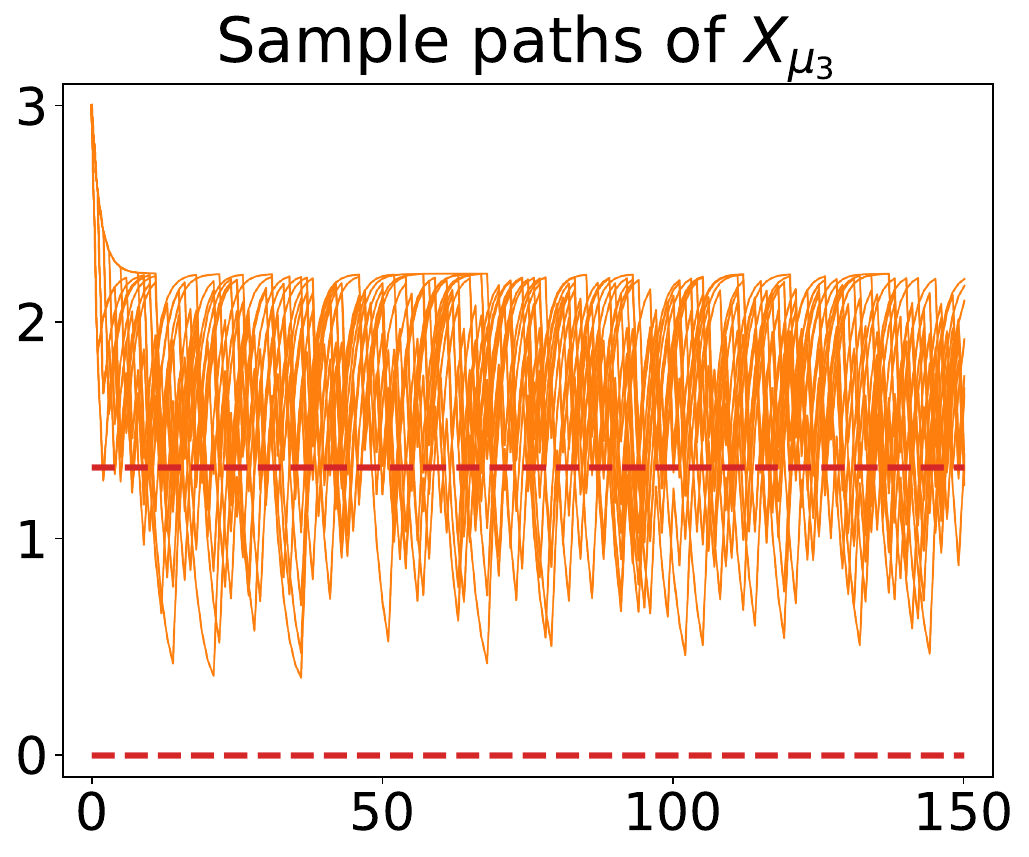}
    \end{minipage}
    \begin{minipage}{0.32\textwidth}
        \centering
        \includegraphics[width=0.95\textwidth]{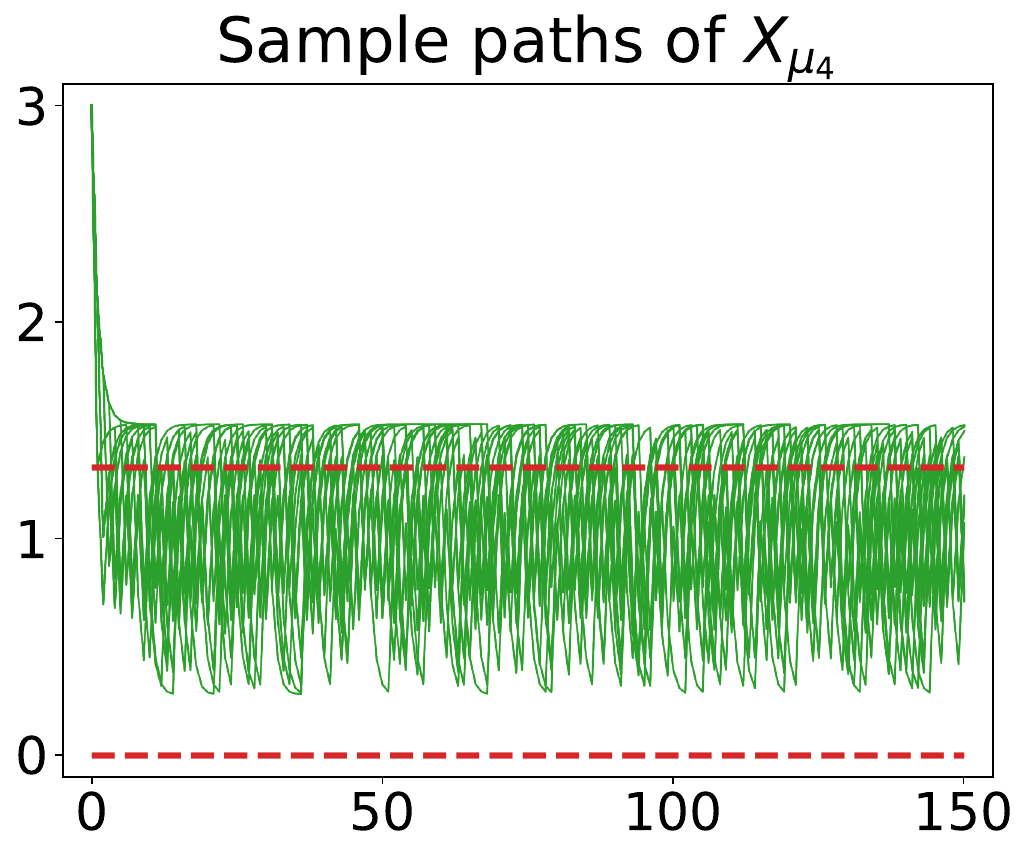}
    \end{minipage}
    \begin{minipage}{0.32\textwidth}
        \centering
        \includegraphics[width=0.95\textwidth]{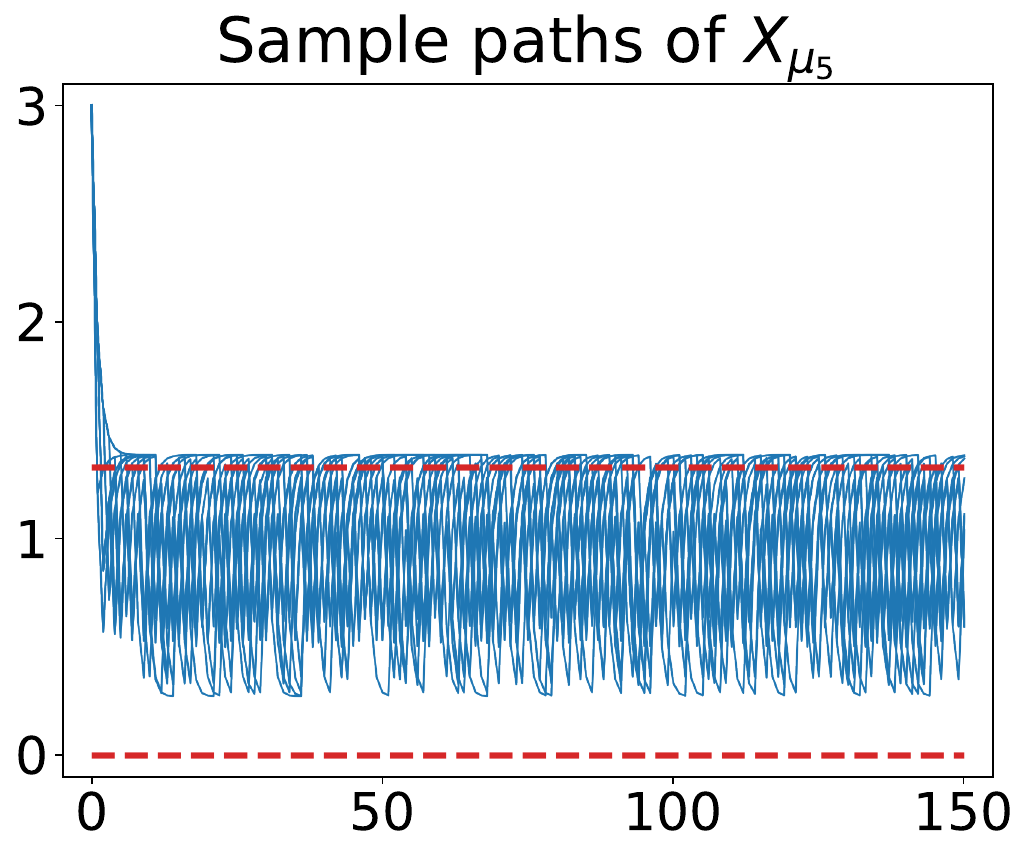}
    \end{minipage}
    \caption{15 sample paths of the closed-loop trajectories $X_{\mu_N}$ for $N=3$ (left), $N=4$ (middle), and $N=5$ (right) together with the upper and lower bound of the realizations of the stationary process (red dashed).}
    \label{fig:stability_paths}
\end{figure}

In Figure~\ref{fig:stability_distribution}, we can see the empirical probability density function and the resulting cumulative distribution function of the closed-loop solution $X_{\mu_N}(150)$ for different horizons $N$. As the horizon $N$ increases, the closed-loop solution converges to the stationary process, now in distribution. The same results hold for the mean and variance of $X_{\mu_N}$, as shown in Figure~\ref{fig:stability_moments}.

\begin{figure}[ht]
    \centering
    \begin{minipage}{0.49\textwidth}
        \centering
        \includegraphics[width=0.7\textwidth]{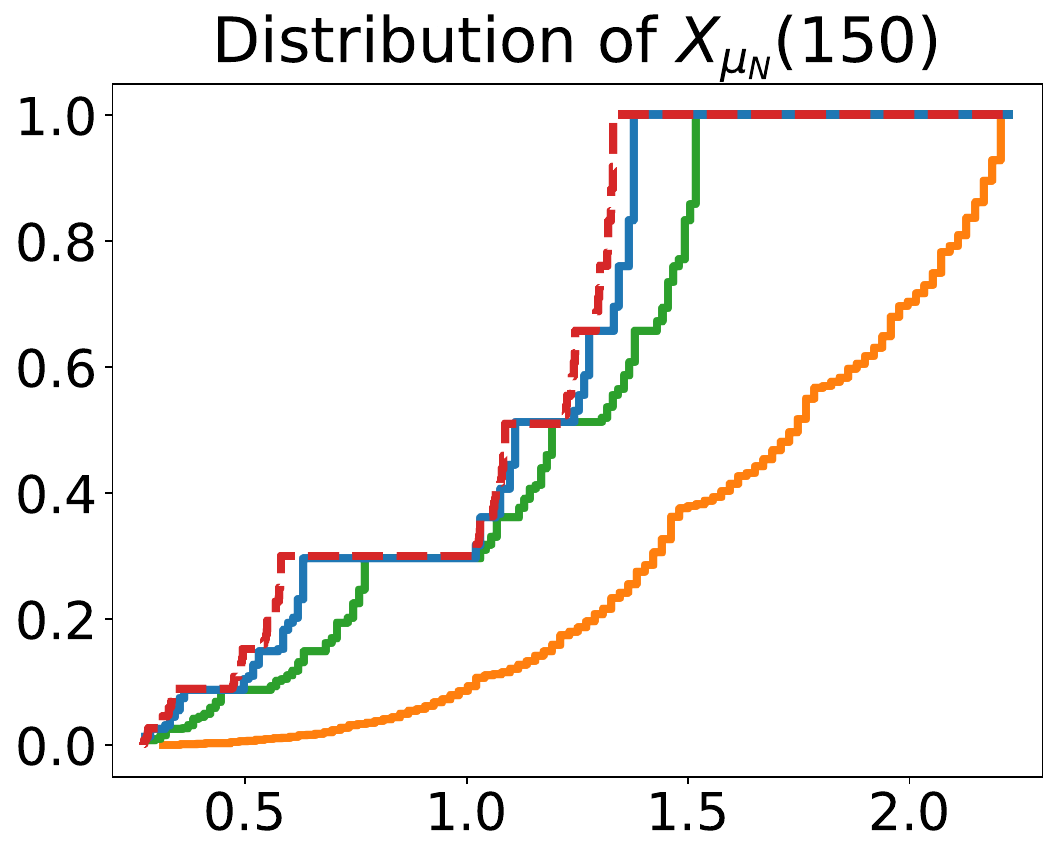}
    \end{minipage}
    \caption{Approximation of the cumulative distribution function of $X_{\mu_N}(150)$ for $N=3$ (orange), $N=4$ (green) and $N=5$ (blue) as well as the approximate stationary distribution (red).}
    \label{fig:stability_distribution}
\end{figure}

\begin{figure}[ht]
    \centering
    \begin{minipage}{0.49\textwidth}
        \centering
        \includegraphics[width=0.7\textwidth]{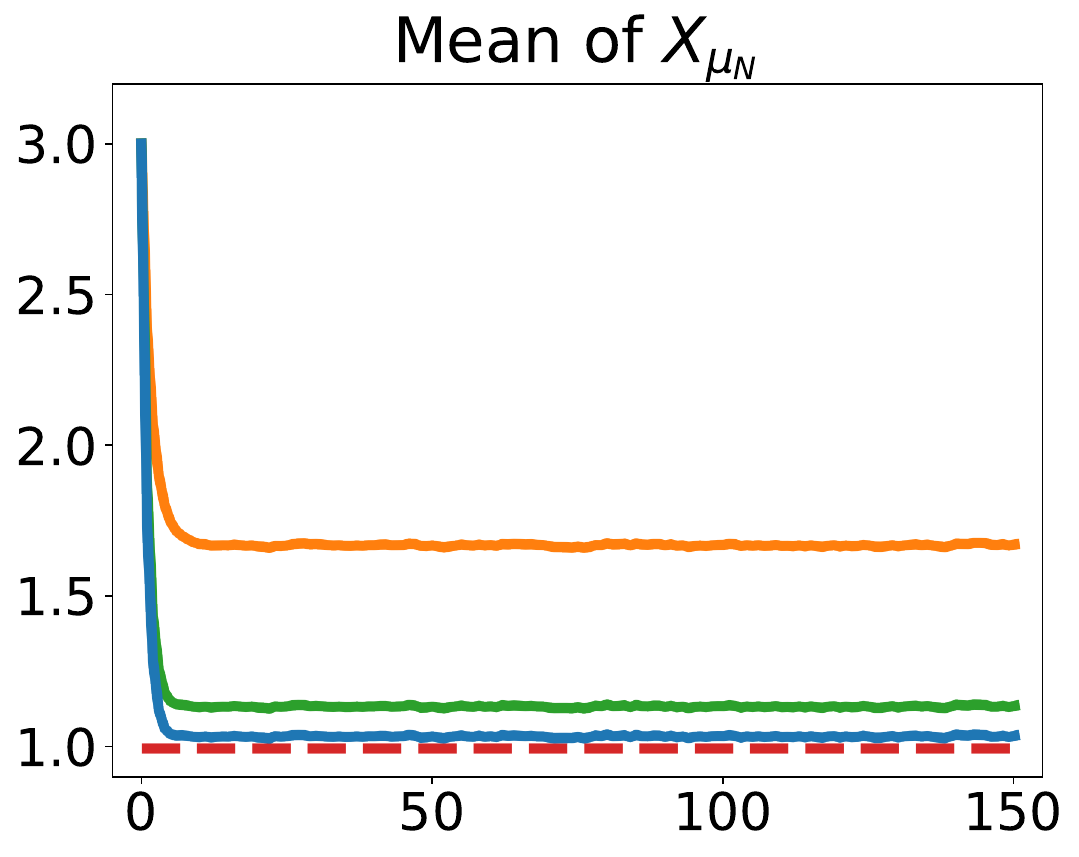}
    \end{minipage}
    \begin{minipage}{0.49\textwidth}
        \centering
        \includegraphics[width=0.7\textwidth]{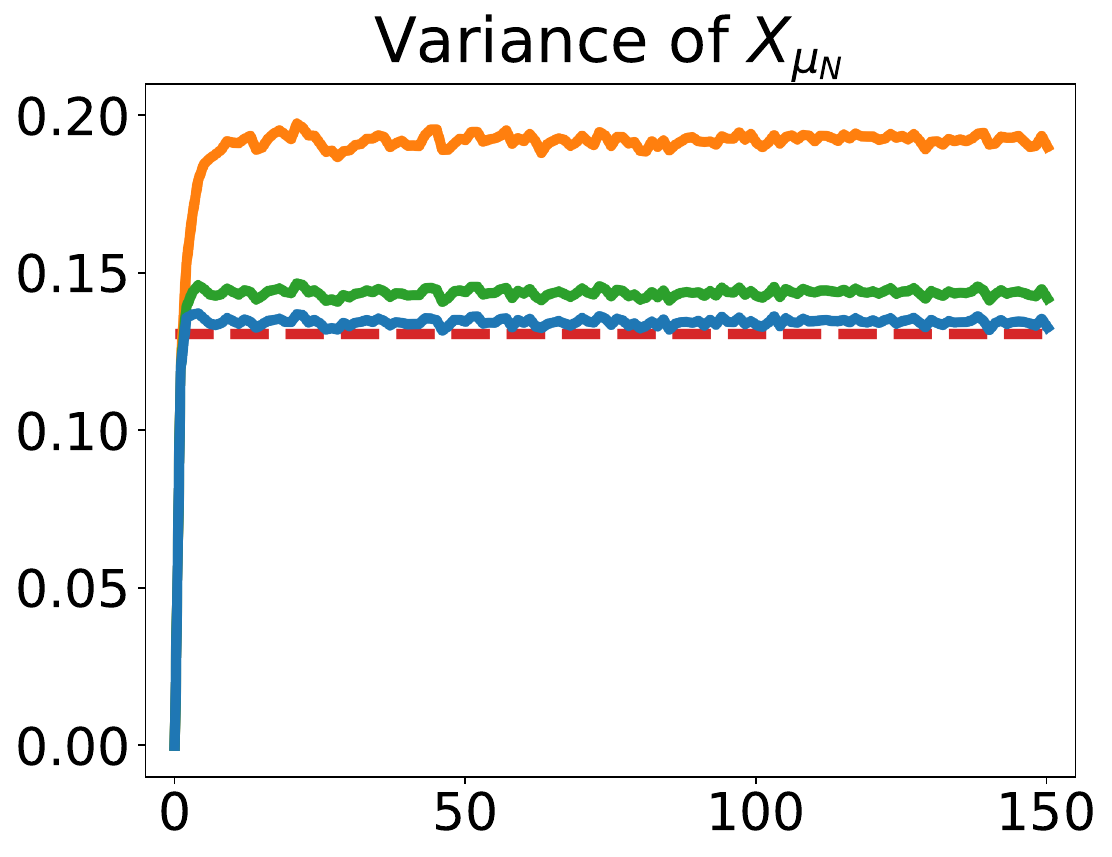}
    \end{minipage}
    \caption{Approximation of the mean (left) and variance (right) of $X_{\mu_N}(150)$ for $N=3$ (orange), $N=4$ (green) and $N=5$ (blue) as well as the approximate stationary values (red).}
    \label{fig:stability_moments}
\end{figure}

In addition to the stability results our findings from Section~\ref{sec:Performance} guarantee averaged and non-averaged performance bounds for the MPC closed loop, which are illustrated in Figure~\ref{fig:costs}.
We can observe that for sufficiently large $K$, the cumulative costs increase approximately linearly at different rates for different horizons $N$, where the difference in the rates is caused by the different $K\nu(N)$ terms from Corollary~\ref{cor:overtaking}. The dependence on $N$ is also visible for the averaged performance, where we can see that the costs converge to a neighborhood of $\ell(\mathbf{X}^s,\mathbf{U}^s)$ and get closer to this value for increasing horizon length, illustrating the results from Theorem~\ref{thm:AvgPerformance}.

\begin{figure}[ht]
    \centering
    \begin{minipage}{0.49\textwidth}
        \centering
        \includegraphics[width=0.7\textwidth]{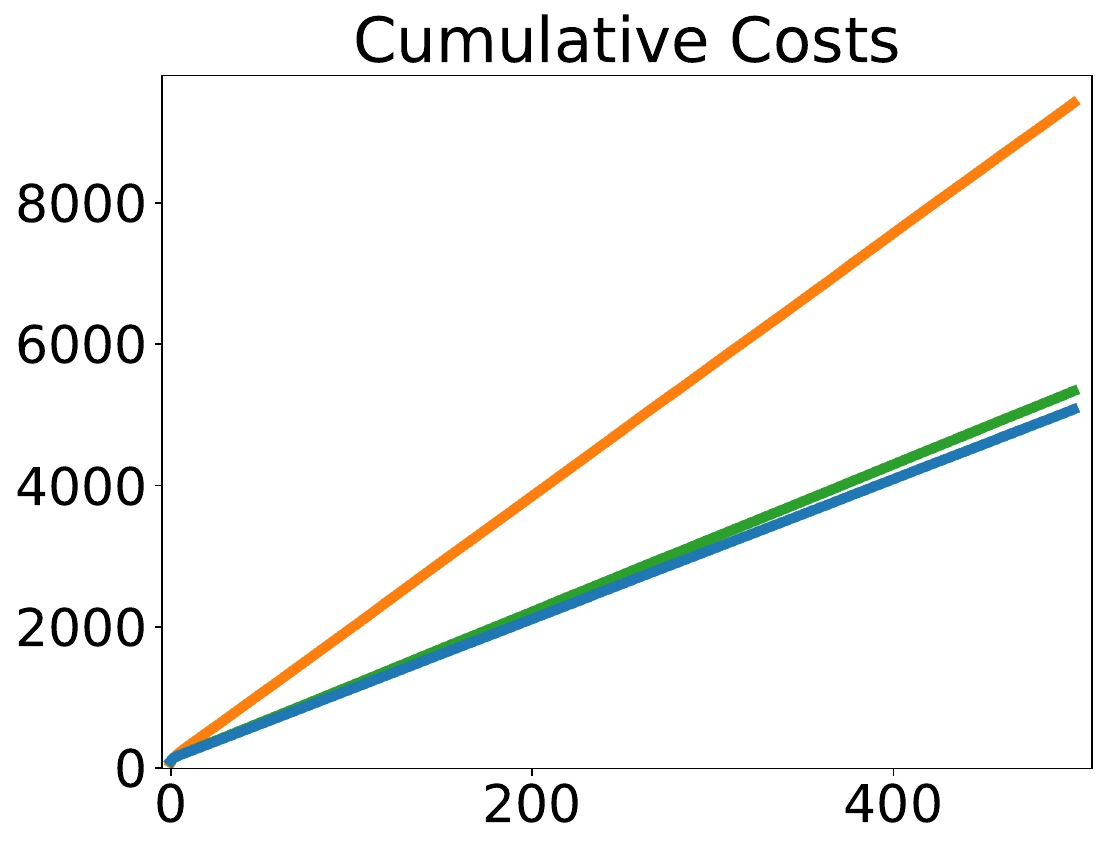}
    \end{minipage}
    \begin{minipage}{0.49\textwidth}
        \centering
        \includegraphics[width=0.7\textwidth]{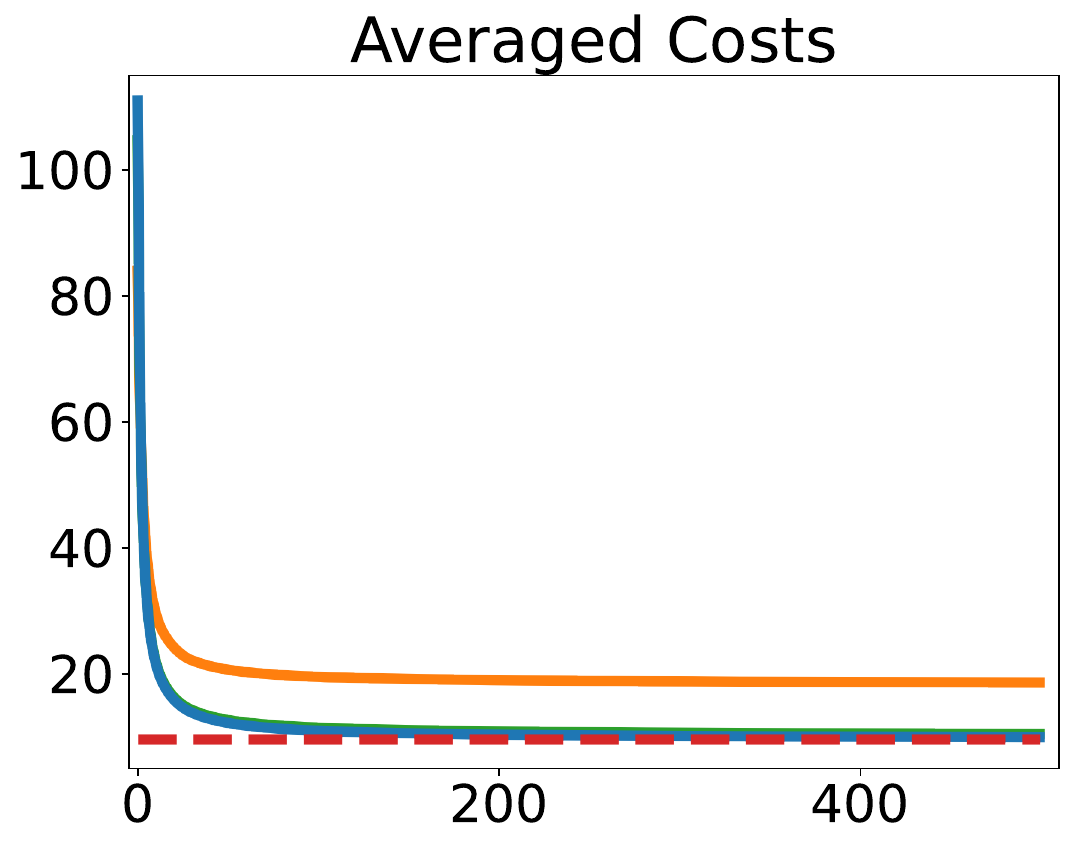}
    \end{minipage}
    \caption{Cumulative costs (left) and stage cost of the stationary process (red) and approximation of the averaged costs (right) for $X_{\mu_N}$ with $N=3$ (orange), $N=4$ (green) and $N=5$ (blue).}
    \label{fig:costs}
\end{figure}

\section{Conclusion} \label{sec:conclusion}
We presented closed-loop results for a stochastic MPC scheme that relies only on state information along a sample path at each prediction step. As such, it is practically implementable for a real-world plant.
Our investigations are based on turnpike and dissipativity concepts for stochastic optimal control problems and the obtained stability properties and performance estimates were illustrated by a nonlinear example.

Future research should focus on the inclusion of probabilistic state and control constraints, such as chance constraints, alternative cost formulations instead of the expectation (e.g., risk measures), and a stability and performance analysis in the case where the open-loop problems cannot be solved exactly due to computational intractability.

\appendix
\section*{Appendix}
\section{Proofs of the preparatory lemmas from Section~\ref{sec:Stability} and Section~\ref{sec:Performance}}  \label{app}
\begin{proof}[Proof of Lemma~\ref{lem:lem1}]
    Let $\tilde{\mathbf{U}}_N^*$ and $\tilde{\mathbf{U}}_{N+1}^*$ be the optimal control sequences of the modified problem~\eqref{eq:stochOCPmodified} on horizon $N$ and $N+1$.
    Then, by optimality and applying the stochastic dynamic programming principle we obtain
    \begin{equation} \label{eq:lem1Equation0}
    \begin{split}
        \tilde{V}_N(\tcl,X_\tcl) & = \tilde{J}_N(\tcl,X_\tcl,\tilde{\mathbf{U}}_N^*)  \\
        &= \tilde{J}_M(\tcl,X_\tcl,\tilde{\mathbf{U}}_N^*) + \tilde{V}_{N-M}
        (\tcl+M,X_{\tilde{\mathbf{U}}_N^*}(M;\tcl,X_\tcl)) \\
        & \leq \tilde{J}_M(\tcl,X_\tcl,\tilde{\mathbf{U}}_{N+1}^*) + \tilde{V}_{N-M}
        (\tcl+M,X_{\tilde{\mathbf{U}}_{N+1}^*}(M;\tcl,X_\tcl))     
    \end{split}
    \end{equation}
	Defining the remainder terms
    \begin{align}
        \tilde{R}_1(\tcl,X_\tcl,M,N) :=& \tilde{V}_{N-M}(\tcl+M,X_{\tilde{\mathbf{U}}_{N}^*}(M;\tcl,X_\tcl))
        - \tilde{V}_{N-M}(\tcl+M,X^s(\tcl+M)) \label{eq:R1} \\
        \tilde{R}_2(\tcl,X_\tcl,M,N) :=& \tilde{V}_{N-M}(\tcl+M,X_{\tilde{\mathbf{U}}_{N+1}^*}(M;\tcl,X_\tcl)) 
        - \tilde{V}_{N-M}(\tcl+M,X^s(\tcl+M)) \nonumber
    \end{align}
    we derive from equation~\eqref{eq:lem1Equation0} that 
    \begin{equation*}
        \tilde{J}_M(\tcl,X_\tcl,\tilde{\mathbf{U}}_N^*) \leq \tilde{J}_M(\tcl,X_\tcl,\tilde{\mathbf{U}}_{N+1}^*) - \tilde{R}_1(\tcl,X_\tcl,M,N) + \tilde{R}_2(\tcl,X_\tcl,M,N)
    \end{equation*}

    Due to Lemma~\ref{lem:TurnpikeModified} we know that there are sets 
    $$\tilde{Q}(\tcl,X_\tcl, L, N) \subseteq \{0,\ldots,N\}$$ 
    and 
    $$\tilde{Q}(\tcl,X_\tcl, L, N+1) \subseteq \{0,\ldots,N+1\}$$ 
    such that 
    \begin{equation*}
        d(X_{\tilde{\mathbf{U}}^*_N}(M;\tcl,X_\tcl),X^s(\tcl+M))  \leq \vartheta(L), ~d(X_{\tilde{\mathbf{U}}^*_{N+1}}(M;\tcl,X_\tcl),X^s(\tcl+M))  \leq \vartheta(L)
    \end{equation*}
    holds for all $M \in \{0,...,N\} \setminus \bar{\mathcal{Q}}(\tcl,X_\tcl,L,N)$ with
    $$\bar{\mathcal{Q}}(\tcl,X_\tcl,L,N) := \tilde{Q}(\tcl,X_\tcl,L,N)  \cup \tilde{Q}(\tcl,X_\tcl,L,N+1)$$
    and $\#\bar{\mathcal{Q}}(\tcl,X_\tcl,L,N) \leq 2L$ since each of the sets $\tilde{Q}$ contains at most $L$ elements.
    
    Now consider $L \geq \vartheta^{-1}(\tilde{r})$.
    Then, we get $\vartheta(L) \leq \tilde{r}$ and thus
    \begin{align*}
        \vert \tilde{R}_1(\tcl,X_\tcl,M,N) \vert & 
        \leq \gamma_{\tilde{V}}(d(X_{\tilde{U}^*_N}(M;\tcl,X_\tcl),X^s(\tcl+M))) \leq \gamma_{\tilde{V}}(\vartheta(L) ) \\
        \vert \tilde{R}_2(\tcl,X_\tcl,M,N) \vert & 
        \leq \gamma_{\tilde{V}}(d(X_{\tilde{U}^*_{N+1}}(M;\tcl,X_\tcl),X^s(\tcl+M))) \leq \gamma_{\tilde{V}}(\vartheta(L) )
    \end{align*}
    for all $M \in \{0,...,N\} \setminus \bar{\mathcal{Q}}(\tcl,X_\tcl,L,N)$.
    
    We continue with repeating the above calculations for $\tilde{J}_M(\tcl,X_\tcl,\tilde{U}_{N+1}^*)$ which yields 
    \begin{equation*}
        \tilde{J}_M(\tcl,X_\tcl,\tilde{\mathbf{U}}_{N+1}^*) \leq \tilde{J}_M(\tcl,X_\tcl,\tilde{\mathbf{U}}_{N}^*) - \tilde{R}_3(\tcl,X_\tcl,M,N) + \tilde{R}_4(\tcl,X_\tcl,M,N)
    \end{equation*}
    with
    \begin{align}
        \tilde{R}_3(\tcl,X_\tcl,M,N) :=& \tilde{V}_{N+1-M}(\tcl+M,X_{\tilde{\mathbf{U}}_{N+1}^*}(M;\tcl,X_\tcl))
        - \tilde{V}_{N+1-M}(\tcl+M,X^s(\tcl+M)) \label{eq:R3} \\
        \tilde{R}_4(\tcl,X_\tcl,M,N) :=& \tilde{V}_{N+1-M}(\tcl+M,X_{\tilde{\mathbf{U}}_{N}^*}(M;\tcl,X_\tcl)) \nonumber 
        - \tilde{V}_{N+1-M}(\tcl+M,X^s(\tcl+M)). \nonumber
    \end{align}
    
    These residuals can be estimated in the same way as $\tilde{R}_1$ and $\tilde{R}_2$, i.e., the inequalities $\vert \tilde{R}_3(\tcl,X_\tcl,M,N) \vert \leq \gamma_{\tilde{V}}(\vartheta(L))$ and $\vert \tilde{R}_4(\tcl,X_\tcl,M,N) \vert \leq \gamma_{\tilde{V}}(\vartheta(L))$ hold for $M \in \bar{\mathcal{Q}}(\tcl,X_\tcl,L,N)$.
    Thus, by defining 
    $$R_1(\tcl,X_\tcl,M,N) := \tilde{J}_M(\tcl,X_\tcl,\tilde{\mathbf{U}}_{N}^*)- \tilde{J}_M(\tcl,X_\tcl,\tilde{\mathbf{U}}_{N+1}^*)$$
    we obtain
    \begin{align*}
        \vert R_1&(\tcl,X_\tcl,M,N) \vert = \vert \tilde{J}_M(\tcl,X_\tcl,\tilde{\mathbf{U}}_{N}^*)- \tilde{J}_M(\tcl,X_\tcl,\tilde{\mathbf{U}}_{N+1}^*)\vert \\
        & \leq \max\{\vert -\tilde{R}_1(\tcl,X_\tcl,M,N) + \tilde{R}_2(\tcl,X_\tcl,M,N) \vert, 
        \vert -\tilde{R}_3(\tcl,X_\tcl,M,N) + \tilde{R}_4(\tcl,X_\tcl,M,N) \vert\} \\
        & \leq \max\{\vert \tilde{R}_1(\tcl,X_\tcl,M,N) \vert + \vert \tilde{R}_2(\tcl,X_\tcl,M,N) \vert, 
        \vert \tilde{R}_3(\tcl,X_\tcl,M,N) \vert + \vert \tilde{R}_4(\tcl,X_\tcl,M,N) \vert \} \\
        & \leq \max\{2\gamma_{\tilde{V}}(\vartheta(L)),2\gamma_{\tilde{V}}(\vartheta(L))\} = 2 \gamma_{\tilde{V}}(\vartheta(L))
    \end{align*}
    for all $M \in \{0,...,N\} \setminus \bar{\mathcal{Q}}(\tcl,X_\tcl,L,N)$, which proves the claim.
\end{proof}

\begin{proof}[Proof of Lemma~\ref{lem:lem2}]
    Using the remainder term $\tilde{R}_1(\tcl,X_\tcl,M,N)$ from equation~\eqref{eq:R1} we obtain
    \begin{equation} \label{eq:lem2Equation0}
    \begin{split}
        \tilde{V}_N(\tcl,X_\tcl) =&  \tilde{J}_M(\tcl,X_\tcl,\tilde{\mathbf{U}}_N^*) + \tilde{V}_{N-M}(\tcl+M,X^s(\tcl+M)) 
        + \tilde{R}_1(\tcl,X_\tcl,M,N) \\
        =& \tilde{J}_M(\tcl,X_\tcl,\tilde{\mathbf{U}}_N^*) + \tilde{R}_1(\tcl,X_\tcl,M,N).
    \end{split}
    \end{equation}
    Here, the last equality follows since Assumption~\ref{ass:modifiedStagecost} implies that $\tilde{V}_K(\tcl,X^s(\tcl)) = 0$ for all $\tcl \in \mathbb{N}_0$ and $K \in \N$.
    Moreover, with an analogous computation we get
    \begin{equation}\label{eq:lem2Equation1}
        \tilde{V}_{N+1}(\tcl,X_\tcl) = \tilde{J}_M(\tcl,X_\tcl,\tilde{\mathbf{U}}_{N+1}^*) + \tilde{R}_3(\tcl,X_\tcl,M,N)
    \end{equation}
    with $\tilde{R}_3(\tcl,X_\tcl,M,N)$ from equation~\eqref{eq:R3}.
    By using Lemma~\ref{lem:lem1} and the identities ~\eqref{eq:lem2Equation0} and \eqref{eq:lem2Equation1}, we thus obtain
    \begin{align*}
         \tilde{V}_{N+1}(\tcl,X_\tcl) &= \tilde{J}_M(\tcl,X_\tcl,\tilde{\mathbf{U}}_{N+1}^*) + \tilde{R}_3(\tcl,X_\tcl,M,N) \\
         &= \tilde{J}_M(\tcl,X_\tcl,\tilde{\mathbf{U}}_{N}^*) + \tilde{R}_3(\tcl,X_\tcl,M,N) - R_1(\tcl,X_\tcl,M,N) \\
         &= \tilde{V}_{N}(\tcl,X_\tcl) - \tilde{R}_1(\tcl,X_\tcl,M,N) + \tilde{R}_3(\tcl,X_\tcl,M,N) - R_1(\tcl,X_\tcl,M,N) 
    \end{align*}
    with $\vert \tilde{R}_1(\tcl,X_\tcl,M,N) \vert \leq \gamma_{\tilde{V}}(\vartheta(L))$, $\vert \tilde{R}_3(\tcl,X_\tcl,M,N) \vert \leq \gamma_{\tilde{V}}(\vartheta(L))$, and $\vert R_1(\tcl,X_\tcl,M,N) \vert \leq 2\gamma_{\tilde{V}}(\vartheta(L))$ for all $M \in \{0,\ldots,N\} \setminus \bar{\mathcal{Q}}(\tcl,X_\tcl,L,N)$, which proves the claim with 
    \begin{equation*}
        R_6(\tcl,X_\tcl,M,N) :=  - \tilde{R}_1(\tcl,X_\tcl,M,N) + \tilde{R}_3(\tcl,X_\tcl,M,N) - R_1(\tcl,X_\tcl,M,N).
    \end{equation*}
\end{proof}

\begin{proof}[Proof of Lemma~\ref{lem:theorem1}]
    The proof of inequality \eqref{eq:thm1} is performed by contradiction.
    To this end, consider an arbitrary $M \leq N$ and assume that there is a control sequence $\mathbf{U} \in \UU_{ad}^M(\tcl,X_\tcl)$ such that 
    \begin{equation} \label{eq:thm1Equation0}
        J_M(\tcl,X_\tcl,\mathbf{U}) + \bar{R}_1(\tcl,X_\tcl,M,N) + \bar{R}_2(\tcl,X_\tcl,M,N) < J_M(\tcl,X_\tcl,\mathbf{U}_{N}^*).
    \end{equation}
    with 
    \begin{align*}
        \bar{R}_1(\tcl,X_\tcl,M,N) & := V_{N-M}(\tcl+M,X_{\mathbf{U}}(M;\tcl,X_\tcl)) 
        - V_{N-M}(\tcl+M,X^s(\tcl+M)) \\
        \bar{R}_2(\tcl,X_\tcl,M,N) & := V_{N-M}(\tcl+M,X^s(\tcl+M)) - V_{N-M}(\tcl+M,X_{\mathbf{U}_N^*}(M;\tcl,X_\tcl)).
    \end{align*}  
    Then, it follows that 
    \begin{align*}
        &\hat{J}_M(\tcl,X_\tcl,\mathbf{U}) + V_{N-M}(\tcl+M,X_{\mathbf{U}}(M;\tcl,X_\tcl)) \\
        & = J_M(\tcl,X_\tcl,\mathbf{U}) + V_{N-M}(\tcl+M,X^s(\tcl+M)) + R_1(\tcl,X_\tcl,M,N) \\
        & =  J_M(\tcl,X_\tcl,\mathbf{U}) + V_{N-M}(\tcl+M,X_{\mathbf{U}_N^*}(M;\tcl,X_\tcl)) 
        + \bar{R}_1(\tcl,X_\tcl,M,N) + \bar{R}_2(\tcl,X_\tcl,M,N) \\
        & <  J_M(\tcl,X_\tcl,\mathbf{U}^*) + V_{N-M}(\tcl+M,X_{\mathbf{U}_N^*}(M;\tcl,X_\tcl)) = V_N(\tcl,X_\tcl)
    \end{align*}
    which contradicts the optimality of $\mathbf{U}_N^*$.
    Thus, we can conclude that 
    \begin{equation} \label{eq:oppositeEquationOfAssumption}
        J_M(\tcl,X_\tcl,\mathbf{U}_{N}^*) \leq J_M(\tcl,X_\tcl,\mathbf{U}) + R_3(\tcl,X_\tcl,M,N)
    \end{equation}
    holds with
    \begin{equation*}
    		R_3(\tcl,X_\tcl,M,N) := \bar{R}_1(\tcl,X_\tcl,M,N) + \bar{R}_2(\tcl,X_\tcl,M,N).
	\end{equation*}    
    It remains to show that the bound on $\vert R_3(\tcl,X_\tcl,M,N) \vert$ holds. 
    By Lemma~\ref{thm:Turnpike} we know that there is a set $\mathcal{Q}(\tcl,X_\tcl,L,N)$ such that $d(\bar{X}_{\mathbf{U}^*}, X^s(\tcl+M)) \leq \vartheta(L)$ for all $M \in \{0,...,N\} \setminus Q(\tcl,X_\tcl,L,N)$.
    Now consider $L \geq \vartheta^{-1}(\eps)$, and $\mathbf{U} \in \UU_{ad}^M(\tcl,X_\tcl)$ with $d(X_{\mathbf{U}}(M;\tcl,X_\tcl),X^s(\tcl+M)) \leq \vartheta(L)$. 
    Then, according to the continuity assumption of the optimal value function, see Assumption~\ref{ass:ContinuityV}, it follows that
    \begin{align*}
        \vert \bar{R}_1(\tcl,X_\tcl,M,N) \vert \leq \gamma_V(N-M,\vartheta(L)) \\
        \vert \bar{R}_2(\tcl,X_\tcl,M,N) \vert \leq \gamma_V(N-M,\vartheta(L))
    \end{align*}
    which implies $\vert R_3(\tcl,X_\tcl,M,N) \vert \leq 2 \gamma_{V}(N-M,\vartheta(L))$ and thus concludes the proof.
\end{proof}

\begin{proof}[Proof of Lemma~\ref{lem:lemma3}]
    Due to the definition of the composite control sequence $\hat{\mathbf{U}}$ it holds that 
    \begin{equation}
        \begin{split} \label{eq:aboveLemma3}
            \tilde{J}_N(\tcl,X_\tcl,\hat{\mathbf{U}}) & = \tilde{J}_M(\tcl,X_\tcl,\mathbf{U}_{N}^*) + \tilde{J}_{N-M}(\tcl+M,\bar{X}_{M},\bar{\mathbf{U}}) \\
            & = \tilde{J}_M(\tcl,X_\tcl,\mathbf{U}_{N}^*) + \tilde{V}_{N-M}(\tcl+M,\bar{X}_{M}) \\
            & = J_M(\tcl,X_\tcl,\mathbf{U}_{N}^*) - M\ell(\mathbf{X}^s,\mathbf{U}^s)
            + \lambda(\tcl,X_\tcl) - \lambda(\tcl+M,\bar{X}_{M}) \\
            &+ \tilde{V}_{N-M}(\tcl+M,\bar{X}_{M}). 
        \end{split}
    \end{equation}
    Now consider $\tilde{R}_1(\tcl,X_\tcl,M,N)$ from equation~\eqref{eq:R1} and
    \begin{align*}
        \hat{R}_1(\tcl,X_\tcl,M,N) & := \tilde{V}_{N-M}(\tcl+M,\bar{X}_{M}) - \tilde{V}_{N-M}(\tcl+M,X^s(\tcl+M)) \\
        \hat{R}_2(\tcl,X_\tcl,M,N) & := \lambda(\tcl+M,X^s(\tcl+M)) -\lambda(\tcl+M,\bar{X}_{M}) \\
        \hat{R}_3(\tcl,X_\tcl,M,N) & := \lambda(\tcl+M,\tilde{X}_M) -\lambda(\tcl+M,X^s(\tcl+M)).
    \end{align*}
    with $\tilde{X}_{M} := X_{\tilde{\mathbf{U}}_N^*}(M;\tcl,X_\tcl)$.
    Then, by Theorem~\ref{lem:theorem1} it holds that
    \begin{align*}
        & J_M(\tcl,X_\tcl,\mathbf{U}_{N}^*) + \lambda(\tcl,X_\tcl) - \lambda(\tcl+M,\bar{X}_{M}) + \tilde{V}_{N-M}(\tcl+M,\bar{X}_{M}) \nonumber\\
        =&  J_M(\tcl,X_\tcl,\mathbf{U}_{N}^*) + \lambda(\tcl,X_\tcl) - \lambda(\tcl+M,X^s(\tcl+M))  \nonumber\\ 
        & + \tilde{V}_{N-M}(\tcl+M,X^s(\tcl+M)) + \hat{R}_1(\tcl,X_\tcl,M,N) + \hat{R}_2(\tcl,X_\tcl,M,N) \\
        = &J_M(\tcl,X_\tcl,\mathbf{U}_{N}^*) + \lambda(\tcl,X_\tcl) - \lambda(\tcl+M,\tilde{X}_M) + \tilde{V}_{N-M}(\tcl+M,\tilde{X}_M) \nonumber\\
        & + \hat{R}_1(\tcl,X_\tcl,M,N) + \hat{R}_2(\tcl,X_\tcl,M,N) + \hat{R}_3(\tcl,X_\tcl,M,N) - \tilde{R}_1(\tcl,X_\tcl,M,N) \\
        \leq& J_M(\tcl,X_\tcl,\tilde{\mathbf{U}}_{N}^*) + \lambda(\tcl,X_\tcl) - \lambda(\tcl+M,\tilde{X}_M) + \tilde{V}_{N-M}(\tcl+M,\tilde{X}_M) \nonumber \\
        & + \hat{R}_1(\tcl,X_\tcl,M,N) + \hat{R}_2(\tcl,X_\tcl,M,N) + \hat{R}_3(\tcl,X_\tcl,M,N) \\
        & - \tilde{R}_1(\tcl,X_\tcl,M,N) + R_3(\tcl,X_\tcl,M,N).
    \end{align*}
    Hence, using the dynamic programming principle we obtain
    \begin{equation*}
    \begin{split}
        \tilde{J}_N(\tcl,X_\tcl,\hat{\mathbf{U}}) \leq& \tilde{J}_M(\tcl,X_\tcl,\tilde{\mathbf{U}}_{N}^*) + \tilde{V}_{N-M}(\tcl+M,\tilde{X}_M) + \hat{R}_4(\tcl,X_\tcl,M,N) \\
        =&  \tilde{V}_N(\tcl,X_\tcl) + \hat{R}_4(\tcl,X_\tcl,M,N) 
    \end{split}
    \end{equation*}
    with
    \begin{equation*}
    \begin{split}
        \hat{R}_4(\tcl,X_\tcl,M,N) :=& \hat{R}_1(\tcl,X_\tcl,M,N) + \hat{R}_2(\tcl,X_\tcl,M,N) + \hat{R}_3(\tcl,X_\tcl,M,N) \\
        &- \tilde{R}_1(\tcl,X_\tcl,M,N) + R_3(\tcl,X_\tcl,M,N)
    \end{split}
    \end{equation*}
    since
    \begin{equation*}
        J_M(\tcl,X_\tcl,\tilde{\mathbf{U}}_{N}^*) - M\ell(\mathbf{X}^s,\mathbf{U}^s) + \lambda(\tcl,X_\tcl) - \lambda(\tcl+M,\tilde{X}_M) = \tilde{J}_M(\tcl,X_\tcl,\tilde{\mathbf{U}}_{N}^*)
    \end{equation*}
    holds.
    Moreover, by the definition of the optimal value function we know that $\tilde{V}_N(\tcl,X_\tcl) \leq \tilde{J}_N(\tcl,X_\tcl,\hat{\mathbf{U}})$ holds, which implies that $\hat{R}_4(\tcl,X_\tcl,M,N) \geq 0$ holds.
    Hence, we can conclude that there exists a residual $R_4(\tcl,X_\tcl,M,N)$ with $0 \leq R_4(\tcl,X_\tcl,M,N) \leq \hat{R}_4(\tcl,X_\tcl,M,N)$ such that
    \begin{equation*}
        \tilde{J}_N(\tcl,X_\tcl,\hat{\mathbf{U}}) = \tilde{V}_N(\tcl,X_\tcl) + R_4(\tcl,X_\tcl,M,N).
    \end{equation*}
    Thus, it remains to show that the upper bound on $\vert R_4(\tcl,X_\tcl,M,N) \vert$ holds.
    For this purpose, consider $L \in \N$ large enough such that $\vartheta(L) \leq \min \{\eps, \tilde{\eps}\}$ holds with $\eps$ form Assumption~\ref{ass:ContinuityV} and $\tilde{\eps}$ from Assumption~\ref{ass:ContinuityModified}.
    Then, using the turnpike properties from Lemma~\ref{thm:Turnpike} and Lemma~\ref{lem:TurnpikeModified}, and the continuity Assumptions~\ref{ass:ContinuityV} and \ref{ass:ContinuityModified} we can conclude that there are sets $\mathcal{Q}(\tcl,X_\tcl,L,N)$ and $\tilde{\mathcal{Q}}(\tcl,X_\tcl,L,N)$ such that
    \begin{align*}
        \vert \hat{R}_1(\tcl,X_\tcl,M,N) \vert \leq \gamma_{\tilde{V}}(\vartheta(L)), 
        \quad \vert \hat{R}_2(\tcl,X_\tcl,M,N) \vert \leq \gamma_{\lambda}(\vartheta(L))
    \end{align*}
    for all $M \in \{0,\ldots,N\} \setminus \mathcal{Q}(\tcl,X_\tcl,L,N)$ and
    \begin{align*}
        \vert \hat{R}_3(\tcl,X_\tcl,M,N) \vert \leq \gamma_{\tilde{V}}(\vartheta(L)), 
        \quad \vert \tilde{R}_1(\tcl,X_\tcl,M,N) \vert \leq \gamma_{\lambda}(\vartheta(L))
    \end{align*}
    for all $M \in \{0,\ldots,N\} \setminus \tilde{\mathcal{Q}}(\tcl,X_\tcl,L,N)$.
    Thus, setting 
    $$\hat{\mathcal{Q}}(\tcl,X_\tcl,L,N) = \mathcal{Q}(\tcl,X_\tcl,L,N) \cup \tilde{\mathcal{Q}}(\tcl,X_\tcl,L,N)$$ 
    and using this estimates and Theorem~\ref{lem:theorem1} we can conclude that
    \begin{equation*}
    \begin{split}
        \vert R_4(\tcl,X_\tcl,M,N) \vert \leq& \vert \hat{R}_4(\tcl,X_\tcl,M,N) \vert \\
        \leq& \vert \hat{R}_1(\tcl,X_\tcl,M,N) \vert + \vert \hat{R}_2(\tcl,X_\tcl,M,N) \vert + \vert \hat{R}_3(\tcl,X_\tcl,M,N) \vert \\
        &+ \vert \tilde{R}_1(\tcl,X_\tcl,M,N) \vert + \vert R_3(\tcl,X_\tcl,M,N) \vert \\
        \leq& 2\gamma_{\tilde{V}}(\vartheta(L)) + 2\gamma_{\lambda}(\vartheta(L))  + 2\gamma_V(N-M,\vartheta(L))
    \end{split}
    \end{equation*}
    holds for all $M \in \hat{\mathcal{Q}}(\tcl,X_\tcl,L,N)$, which shows the claim, since each of the sets $\mathcal{Q}$ and $\tilde{\mathcal{Q}}$ contains at most $L$ elements.
\end{proof}

\begin{proof}[Proof of Lemma~\ref{lem:optimalOperation}]
    Because of the dissipativity Assumption~\ref{ass:Dissipativity} we know that there exists a uniform lower bound $-C^l_{\lambda} < 0$ on $\lambda$ such that
        \begin{equation*}
        \begin{split}
            &\frac{1}{K} \sum_{\tcl=0}^{K-1} \ell(X_{\mathbf{U}}(\tcl;0,X_0), U(\tcl) \\
            \geq& \frac{1}{K} \sum_{\tcl=0}^{K-1} \ell(\mathbf{X}^s,\mathbf{U}^s) - \lambda(\tcl,X_{\mathbf{U}}(\tcl;0,X_0)) + \lambda(\tcl+1,X_{\mathbf{U}}(\tcl+1;0,X_0)) \\
            \geq& \ell(\mathbf{X}^s,\mathbf{U}^s) - \frac{\lambda(0,X_0)}{K} - \frac{C^l_{\lambda}}{K}
        \end{split}
        \end{equation*}
        holds for all $X_0 \in \RR{\Omega,\X}$ and $\mathbf{U} \in \UU_{ad}^K(X_0)$, which proves the claim by letting $K$ go to infinity.
\end{proof}

{\bibliographystyle{abbrv} 
  \bibliography{references} 
}

\end{document}